\documentclass[psamsfonts]{amsart}
\usepackage{amssymb}
\usepackage{graphicx,color}
\usepackage{amssymb,amsfonts}
\usepackage[all,arc]{xy}
\usepackage{enumerate}
\usepackage{mathrsfs}

\newtheorem{thm}{Theorem}[section]
\newtheorem{lemma}[thm]{Lemma}

\theoremstyle{definition}

\theoremstyle{remark}
\newtheorem{rem}[thm]{Remark}
\makeatletter
\let\c@equation\c@thm
\makeatother
\numberwithin{equation}{section}

\title[Convex Hypersurfaces of Prescribed Curvature ]{ Convex Hypersurfaces with Prescribed Scalar Curvature and Asymptotic Boundary in Hyperbolic Space}

\thanks{The author is supported by the grant (no. AUGA5710000618) from Harbin Institute of Technology and
National Natural Science Foundation of China (No. 12001138).}

\author{Zhenan Sui}

\address{Institute for Advanced Study in Mathematics of HIT, Harbin Institute of Technology, Harbin, China}
\email{sui.4@osu.edu}

\begin{document}

\begin{abstract}
The existence of a smooth complete strictly locally convex hypersurface with prescribed scalar curvature and asymptotic boundary at infinity in $\mathbb{H}^{3}$ is proved under the assumption that there exists a strictly locally convex subsolution.
\end{abstract}

\maketitle


\section {\large Introduction}

\vspace{6mm}

In this paper, we are concerned with the asymptotic Plateau type problem in hyperbolic space $\mathbb{H}^{n+1}$: to find a complete strictly locally convex hypersurface $\Sigma$ with prescribed curvature and asymptotic boundary at infinity. For hyperbolic space, we will use the half-space model
\[\mathbb{H}^{n+1} = \,\{ (x, x_{n+1}) \in \mathbb{R}^{n+1} \,\big\vert\,x = (x_1, \ldots, x_n) \in \mathbb{R}^n, \, x_{n+1} > 0 \} \]
equipped with the hyperbolic metric
\[ d s^2 = \frac{1}{x_{n+1}^2} \,\sum_{i = 1}^{n+1} \, d x_i^2. \]
The ideal boundary at infinity of $\mathbb{H}^{n+1}$ can be identified with
\[\partial_{\infty} \mathbb{H}^{n+1} = \mathbb{R}^n = \mathbb{R}^n \times \{0\} \,\subset \mathbb{R}^{n+1}\]
and the asymptotic boundary $\Gamma$ of $\Sigma$ is given at $\partial_{\infty} \mathbb{H}^{n+1}$, which consists of a disjoint collection of smooth closed embedded $(n - 1)$ dimensional submanifolds $\{ \Gamma_1, \ldots, \Gamma_m \}$. Given a positive function
$\psi \in C^{\infty} (\mathbb{H}^{n+1})$,
we are interested in finding a complete strictly locally convex hypersurfaces $\Sigma$ in $\mathbb{H}^{n+1}$ satisfying the curvature equation
\begin{equation}  \label{eqn1}
f (\kappa) = \,\sigma_k^{1/k} ( \kappa ) = \psi^{1/k}( {\bf x} )
\end{equation}
as well as with the asymptotic boundary
\begin{equation} \label{eqn2}
\partial \Sigma = \Gamma,
\end{equation}
where ${\bf x}$ is a conformal Killing field which will be specified in section 6, $\kappa = (\kappa_1, \ldots, \kappa_n)$ are the hyperbolic principal curvatures of $\Sigma$ at ${\bf x}$, and
\[\sigma_k (\lambda) = \sum\limits_{1 \leq i_1 < \ldots < i_k \leq n} \,\lambda_{i_1} \cdots \lambda_{i_k}\]
is the $k$-th elementary symmetric function defined on $k$-th G\r arding's cone
\[\Gamma_k \equiv \{ \lambda \in \mathbb{R}^n \vert \, \sigma_j (\lambda) > 0,\,\, j = 1, \ldots, k \}. \]
$\sigma_k (\kappa)$ is the so called $k$-th Weingarten curvature of $\Sigma$. In particular, the $1$st, $2$nd and $n$-th Weingarten curvature correspond to mean curvature, scalar curvature and Gauss curvature respectively. We call a hypersurface $\Sigma$ strictly locally convex (locally convex) if all principal curvatures at any point of $\Sigma$ are positive (nonnegative).

In this paper, all hypersurfaces are assumed to be connected and orientable.
We will see from Lemma \ref{LemmaV} that a strictly locally convex hypersurface in $\mathbb{H}^{n+1}$ with compact (asymptotic) boundary  must be a vertical graph over a bounded domain in $\mathbb{R}^n$. We thus assume the normal vector field on $\Sigma$ to be upward. Write
\[ \Sigma =\, \{ ( x, \,u(x) ) \in \mathbb{R}^{n+1}_+  \,\big\vert \, x \in \Omega \}, \]
where $\Omega$ is the bounded domain on $\partial_{\infty} \mathbb{H}^{n+1} = \mathbb{R}^n$ enclosed by $\Gamma$. Consequently, \eqref{eqn1}--\eqref{eqn2} can be expressed in terms of $u$,
\begin{equation} \label{eqn9}
\left\{ \begin{aligned} f ( \kappa [\, u \,] )\, =  & \,\,\, \psi^{\frac{1}{k}}(x,\, u) \quad\quad & \mbox{in} \quad \Omega, \\ u \,  = & \,\,\, 0 \quad \quad & \mbox{on} \quad \Gamma. \end{aligned} \right.
\end{equation}

The essential difficulty for the Plateau type problem \eqref{eqn9} is due to the singularity at $u = 0$. When $\psi$ is a positive constant, problem \eqref{eqn9} has been extensively investigated in \cite{GS00, GSS09, GS10, GS11, GSX14} (see also the references therein for some previous work). Their basic idea is: first, to prove the existence of a solution $u^{\epsilon}$ to the approximate Dirichlet problem
\begin{equation} \label{eqn11}
\left\{ \begin{aligned} f ( \kappa [\, u \,] )\, =  & \,\,\, \psi^{\frac{1}{k}}(x,\, u) \quad\quad & \mbox{in} \quad \Omega, \\ u \,  = & \,\,\, \epsilon \quad \quad & \mbox{on} \quad \Gamma, \end{aligned} \right.
\end{equation}
and then, to show these $u^{\epsilon}$ converge to a solution of \eqref{eqn9} after passing to a subsequence.
For general $\psi$,  Szapiel \cite{Sz05} studied the existence of strictly locally convex solutions to \eqref{eqn11} for $f = \sigma_n^{1/n}$, but he also assumed a very strong assumption on $f$ (see (1.11) in \cite{Sz05}) which excluded the case $f = \sigma_n^{1/n}$. As far as the author knows, there is no literature which gives an existence result for the asymptotic Plateau type problem \eqref{eqn9} for general $\psi$.

Our first task in this paper is to improve the result of \cite{Sz05}.
As in \cite{GS04}, we assume the existence of a strictly locally convex  subsolution $\underline{u} \in C^4(\Omega)$, that is,
\begin{equation} \label{eqn8}
\left\{ \begin{aligned} f ( \kappa [\, \underline{u} \,] )\, \geq  & \,\,\, \psi^{\frac{1}{k}}(x, \,\underline{u} ) \quad\quad & \mbox{in} \quad \Omega, \\ \underline{u} \,  = & \,\,\, 0 \quad \quad & \mbox{on} \quad \Gamma. \end{aligned} \right.
\end{equation}
Different from \cite{GSS09, GS10, GS11, GSX14, Sz05}, we take a new approximate Dirichlet problem
\begin{equation} \label{eqn10}
\left\{\begin{aligned} f ( \kappa [\, u \,] )\, =  & \,\,\, \psi^{\frac{1}{k}}(x,\, u) \quad\quad & \mbox{in} \quad \Omega_{\epsilon}, \\ u \,  = & \,\,\, \epsilon \quad \quad & \mbox{on} \quad \Gamma_{\epsilon}, \end{aligned} \right.
\end{equation}
where the $\epsilon$-level set of $\underline{u}$ and its enclosed region in $\mathbb{R}^n$ are respectively
\[ \Gamma_{\epsilon} = \,\{ x \in \Omega \,\big\vert\,\, \underline{u}(x) \, = \,\epsilon\, \} \quad \mbox{and} \quad
\Omega_{\epsilon} = \,\{ x \in \Omega \,\big\vert\,\, \underline{u}(x) \, > \,\epsilon \, \}. \]
We may assume the dimension of $\Gamma_{\epsilon}$ is $(n - 1)$ by Sard's theorem, and in addition, $\Gamma_{\epsilon} \in C^4$.

A crucial step for proving the existence of a strictly locally convex solution to \eqref{eqn10} is to establish second order a priori estimates for strictly locally convex solutions $u$ of \eqref{eqn10} satisfying $u \geq \underline{u}$ on $\Omega_{\epsilon}$. An essential difference from \cite{GSS09, GS10, GS11, GSX14} is that we allow  the $C^2$ bound to depend on $\epsilon$. This looser requirement gives us more flexibility to apply techniques for general Dirichlet problem and with less technical assumptions (for example, there is no prescribed upper bound for $\psi$). For $C^2$ boundary estimates, we change the variable from $u$ to $v$ by $u = \sqrt{v}$ (see \cite{Sui18} for a similar idea for radial graphs), which is the main difference from \cite{GSS09, Sz05} and fundamentally improves the result in \cite{Sz05}.

One reason that we purely study strictly locally convex hypersurfaces is due to $C^2$ boundary estimates.  In \cite{GS10}, Guan-Spruck assumed $\Gamma$ to be mean convex. Then the solution $u$ behaves nicely near $\Gamma$ and therefore $k$-admissible solutions can be studied in their framework. However, without any geometric assumptions on $\Gamma_{\epsilon}$, $C^2$ boundary estimates can only be obtained for strictly locally convex hypersurfaces.

In order to apply continuity method and degree theory to prove the existence of a strictly locally convex solution to \eqref{eqn10}, the strict local convexity has to be preserved during the continuity process. This is true when $k = n$ in view of the nondegeneracy of \eqref{eqn10}, while for $1 \leq k < n$, we have to impose certain assumptions on $\Omega$, $\underline{u}$ and $\psi$ to guarantee the full rank of the second fundamental form on locally convex $\Sigma$ up to the boundary. In this paper, we want to apply the constant rank theorem developed in \cite{KL87, GM03, GLM06} to Dirichlet boundary value problems when assuming a subsolution. For this, we assume
\begin{equation} \label{eqn12}
\left\{ \Big( \frac{\underline{u}}{f(\kappa [\underline{u}])} \Big)_{x_{\alpha} x_{\beta}} \right\}_{n \times n} \,\geq 0,
\end{equation}
\begin{equation}  \label{eqn13}
 \left(
         \begin{array}{cc}
           \frac{k+1}{k} \frac{\psi_{x_\alpha} \psi_{x_\beta}}{\psi} - \psi_{x_{\alpha} x_{\beta}}  - \frac{k \psi}{u^2} \delta_{\alpha\beta} + \frac{\psi_u}{u} \delta_{\alpha\beta}
            & \frac{k+1}{k} \frac{\psi_{x_\alpha} \psi_{u}}{\psi} - \psi_{x_{\alpha} u} - \frac{\psi_{x_\alpha}}{u} \\
           \frac{k+1}{k} \frac{\psi_{x_\alpha} \psi_{u}}{\psi} - \psi_{x_{\alpha} u} - \frac{\psi_{x_\alpha}}{u} &
          \frac{k+1}{k} \frac{\psi_{u}^2}{\psi} - \psi_{u u} - \frac{k \,\psi}{u^2} - \frac{\psi_u}{u} \\
         \end{array}
       \right) \,\geq 0.
\end{equation}
Besides, we also need a condition which can guarantee that locally convex solutions to the associated equations of \eqref{eqn10} are strictly locally convex near the boundary $\Gamma_{\epsilon}$. However, we did not find such a condition. Therefore, our existence results are limited to $k = n$.

\begin{thm} \label{Theorem1}
Under the subsolution condition \eqref{eqn8}, for $k = n$,
there exists a smooth strictly locally convex solution $u^{\epsilon}$ to the Dirichlet problem \eqref{eqn10} with $u^{\epsilon}  \geq \underline{u}$ in $\Omega_{\epsilon}$.
\end{thm}

Our second task in this paper is to solve \eqref{eqn9}. A central issue is to provide certain uniform $C^2$ bound for $u^\epsilon$. Different from \cite{GSS09, GS10, GS11, GSX14}, where the authors derived uniform bound for certain quantities regarding solutions of \eqref{eqn11} under certain assumptions,  we use \eqref{eqn10} as an approximate Dirichlet problem and tolerate the $\epsilon$-dependent $C^2$ bound for solutions to \eqref{eqn10}, since we are able to use the idea of Guan-Qiu \cite{GQ17}, who established $C^2$ interior estimates for convex hypersurfaces with prescribed scalar curvature in $\mathbb{R}^{n+1}$. We extend their estimates to $\mathbb{H}^{n+1}$, which, together with Evans-Krylov interior estimates (see \cite{Evans, Krylov}) and standard diagonal process, lead to the following existence result.  Since the pure $C^2$ interior estimates can only be derived up to scalar curvature equations (see Pogorelov \cite{Po78} and Urbas \cite{Ur90} for counterexamples when $k \geq 3$), we hope to investigate the cases $k \geq 3$ in future work by other means. Meanwhile, interior $C^2$ estimates are limited to hypersurfaces satisfying certain convexity property (see \cite{GQ17}), which also explains why we only focus on strictly locally convex hypersurfaces.

\begin{thm} \label{Theorem2}
In $\mathbb{H}^3$, for $f = \sigma_2^{1/2}$, under the subsolution condition \eqref{eqn8},
there exists a smooth strictly locally convex solution $u \geq \underline{u}$ to \eqref{eqn9} on $\Omega$, equivalently, there exists a smooth complete strictly locally convex vertical graph solving \eqref{eqn1}--\eqref{eqn2}.
\end{thm}

This paper is organized as follows: in section 2, we provide some basic formulae, properties and calculations for vertical graphs. The $C^2$ estimates for strictly locally convex solutions of \eqref{eqn10} are presented in section 3 and 4. In section 5, we prove Theorem \ref{Theorem1} via continuity method and degree theory. Section 6 provides the interior $C^2$ estimates for convex solutions to prescribed scalar curvature equations in $\mathbb{H}^{n+1}$, which finishes the proof of Theorem \ref{Theorem2}.

\medskip
\noindent
{\bf Acknowledgements} \quad
The author would like to thank Dr. Zhizhang Wang and Dr. Wei Sun for many useful and enlightening discussions. The author also wish to express the deep thanks to the reviewer, who pointed out a mistake in the previous version and gave many helpful suggestions, which help the author have a better understanding of the problem.

\vspace{5mm}

\section{ \large Vertical graphs }

\vspace{4mm}

Suppose $\Sigma$ is locally represented as the graph of a positive $C^2$ function over a domain $\Omega \subset \mathbb{R}^n$:
\[ \Sigma =\, \{ ( x, \,u(x) ) \in \mathbb{R}^{n+1}_+  \,\big\vert \, x \in \Omega \}. \]
Since the coordinate vector fields on $\Sigma$ are
\[  \partial_i + u_i \,\partial_{n + 1}, \quad\quad i = 1, \ldots, n  \quad \mbox{where} \quad \partial_i = \frac{\partial}{\partial x_i}, \]
thus the upward Euclidean unit normal vector field to $\Sigma$, the Euclidean metric, its inverse and the Euclidean second fundamental form of $\Sigma$ are given respectively by
\[\nu = \Big( \frac{- D u}{w},\, \frac{1}{w} \Big), \quad\quad w = \sqrt{ 1 + |D u |^2},  \]

\[ \tilde{g}_{ij} = \delta_{ij} + u_i u_j, \quad\quad  \tilde{g}^{ij} =  \delta_{ij} - \frac{u_i u_j}{w^2},  \quad\quad  \tilde{h}_{ij} = \frac{u_{ij}}{w}. \]
Consequently, the Euclidean principal curvatures $\tilde{\kappa} [ \Sigma ]$ are the eigenvalues of the symmetric matrix:

\[ \tilde{a}_{ij} := \frac{1}{w} \gamma^{ik} u_{kl} \gamma^{lj}, \]
where
\begin{equation*}
 \gamma^{ik} = \delta_{ik} - \frac{u_i u_k}{w ( 1 + w )}
\end{equation*}
and its inverse
\[ \gamma_{ik} = \delta_{ik} + \frac{u_i u_k}{1 + w}, \quad\quad  \gamma_{ik} \gamma_{kj} = \tilde{g}_{ij}. \]

For geometric quantities in hyperbolic space,  we first note that
the upward hyperbolic unit normal vector field to $\Sigma$ is
\[ {\bf n} = u \,\nu = \, u \,\Big( \frac{- D u}{w}, \,\,\frac{1}{w} \Big) \]
and the hyperbolic metric of $\Sigma$ is
\begin{equation}  \label{eq0-7}
g_{ij} = \frac{1}{u^2}\, ( \delta_{ij} + u_i u_j ).
\end{equation}
To compute the hyperbolic second fundamental form $h_{ij}$ of $\Sigma$,
applying the Christoffel symbols in $\mathbb{H}^{n + 1}$,
\begin{equation} \label{eq2-2}
{\bf\Gamma}_{ij}^k = \,\frac{1}{x_{n + 1}} \big(- \delta_{ik} \delta_{n + 1\, j} - \delta_{kj} \delta_{n + 1 \,i} + \delta_{k\, n + 1} \delta_{ij} \big),
\end{equation}
we obtain
\begin{equation*}
{\bf D}_{\partial_i + u_i \partial_{n + 1}} \big( \partial_j + u_j \,\partial_{n + 1} \big) = \,  - \frac{u_j}{x_{n + 1}}\, \partial_i - \frac{u_i}{x_{n + 1}} \,\partial_j + \Big( \frac{\delta_{ij}}{x_{n + 1}} + u_{ij} - \frac{u_i u_j}{x_{n + 1}} \Big)\, \partial_{n + 1},
\end{equation*}
where ${\bf D}$ denotes the Levi-Civita connection in $\mathbb{H}^{n+1}$.  Therefore,
\[ h_{ij} = \frac{1}{u^2 w} ( \delta_{ij} + u_i u_j + u u_{ij} ). \]
The hyperbolic principal curvatures $\kappa [ \Sigma ]$ are the eigenvalues of the symmetric matrix $A [u] = \{ a_{ij} \}$:
\begin{equation*}
 a_{ij} = \,  u^2 \gamma^{ik} h_{kl} \gamma^{lj}
       = \,\frac{1}{w} \,\gamma^{ik} ( \delta_{kl} + u_k u_l + u u_{kl} ) \,\gamma^{lj} = \,  \frac{1}{w} ( \delta_{ij} + u \gamma^{ik} u_{kl} \gamma^{lj} ).
\end{equation*}

\begin{rem} \label{rmk1}
The graph of $u$ is strictly locally convex if and only if  the symmetric matrix $\{ a_{ij }\}$, $\{ h_{ij} \}$ or $\{ \delta_{ij} + u_i u_j + u u_{ij} \}$ is positive definite.
\end{rem}

\vspace{2mm}

\begin{rem}
From the above discussion, we can see that
\begin{equation} \label{eq0-3}
h_{ij} = \frac{1}{u}\,\tilde{h}_{ij} + \frac{\nu^{n+1}}{u^2}\, \tilde{g}_{ij},
\end{equation}
where $\nu^{n+1} = \nu \cdot \partial_{n + 1}$ and $\cdot$ is the inner product in $\mathbb{R}^{n+1}$.
This formula indeed holds for any local frame on any hypersurface $\Sigma$
(which may not be a graph). The relation between $\kappa[\Sigma]$ and $\tilde{\kappa} [\Sigma]$ is
\begin{equation} \label{eqn14}
\kappa_i = \, u \,\tilde{\kappa_i} + \nu^{n+1}, \quad \quad i = 1, \ldots, n.
\end{equation}
\end{rem}

\vspace{2mm}

We observe the following phenomenon for strictly locally convex hypersurfaces in $\mathbb{H}^{n+1}$ (see also Lemma 3.3 in \cite{GSS09} for a similar assertion).

\vspace{2mm}

\begin{lemma} \label{LemmaV}
Let $\Sigma$ be a connected, orientable, strictly locally convex hypersurface in $\mathbb{H}^{n+1}$ with a specially chosen orientation. Then $\Sigma$ must be a vertical graph.
\end{lemma}

\begin{proof}
Suppose $\Sigma$ is not a vertical graph. Then there exists a vertical line (of dimension $1$) intersecting $\Sigma$ at two distinct points $p_1$ and $p_2$. Since $\Sigma$ is orientable, we may assume that $\nu^{n+1} (p_1) \cdot \nu^{n+1} (p_2) \leq 0$. Since $\Sigma$ is connected, there exists a $1$-dimensional curve $\gamma$ on $\Sigma$ connecting $p_1$ and $p_2$. Among the tangent hyperplanes (of dimension $n$) to $\Sigma$ along $\gamma$, choose a vertical one which is tangent to $\Sigma$ at a point $p_3$.  At $p_3$, $\nu^{n+1} = 0$ and $u > 0$. By \eqref{eqn14}, $\tilde{\kappa}_i > 0$ for all $i$ at $p_3$. On the other hand, let $P$ be a $2$-dimensional plane passing through $p_1$, $p_2$ and $p_3$. If $P \cap \Sigma$ is $1$-dimensional and has nonpositive (Euclidean) curvature at $p_3$ with respect to $\nu$, we reach a contradiction; otherwise we take a different orientation of $\Sigma$, then $\Sigma$ is either not strictly locally convex or we reach a contradiction. If $P \cap \Sigma$ is $2$-dimensional, then any line on $P \cap \Sigma$ through $p_3$ leads to a contradiction.
\end{proof}

Equation \eqref{eqn1} can be written as
\begin{equation} \label{eq1-1}
  f( \kappa [\, u \,] ) = f( \lambda( A[ \, u \,] )) = F( A[ \,u\, ] ) = \, \psi^{1/k} ( x,\, u ).
\end{equation}
Recall that the curvature function $f$ satisfies the fundamental structure conditions
\begin{equation} \label{eqn3}
f_i (\lambda) \equiv \frac{\partial f(\lambda)}{\partial \lambda_i} > 0 \quad \mbox{in} \,\,  \Gamma_k,\quad i = 1, \ldots, n,
\end{equation}
\begin{equation} \label{eqn4}
f \,\,\mbox{is} \,\,\mbox{concave}\,\, \mbox{in} \,\, \Gamma_k,
\end{equation}
\begin{equation} \label{eqn5}
f > 0 \quad \mbox{in} \,\, \Gamma_k, \quad\quad f = 0 \quad \mbox{on} \,\, \partial\Gamma_k.
\end{equation}

\vspace{5mm}

\section{\large Second Order Boundary Estimates}

\vspace{5mm}

In this section and the next section, we derive a priori $C^2$ estimates for strictly locally convex solution $u$ to the Dirichlet problem \eqref{eqn10} with $u \geq \underline{u}$ in $\Omega_{\epsilon}$. By Evans-Krylov theory \cite{Evans,Krylov}, classical continuity method and degree theory (see \cite{Li89}) we prove the existence of a strictly locally convex solution to \eqref{eqn10}. Higher-order regularity then follows from classical Schauder theory.

Let $u \geq \underline{u}$ be a strictly locally convex function over $\Omega_{\epsilon}$ with $u = \underline{u}$ on $\Gamma_{\epsilon}$. We have the following $C^0$ estimate:
\begin{equation} \label{eqn15}
 \underline{u}\,\, \leq u \leq \sqrt{\epsilon^2 + (\mbox{diam} \Omega)^2} \quad \mbox{in} \quad \overline{\Omega_{\epsilon}}.
\end{equation}
In fact, by Remark \ref{rmk1}, for any $x_0 \in \Omega_{\epsilon}$,  the function $u^2 + |x - x_0|^2$ is Euclidean strictly locally convex in $\Omega_{\epsilon}$, over which, we have
\[  u^2 \leq u^2 + |x - x_0|^2 \leq \max\limits_{\Gamma_{\epsilon}} ( u^2 + |x - x_0|^2) \leq \epsilon^2 + (\mbox{diam}\Omega)^2.\]
Therefore we obtain \eqref{eqn15}.

For the gradient estimate, we perform a transformation $u = \sqrt{v}$. Denote
\[W = \sqrt{4 v + |D v|^2}. \]
The geometric quantities in section 2 can be expressed in terms of $v$,
\[ \begin{aligned} \gamma^{ik} = \delta_{ik} - \frac{v_i v_k}{W (2 \sqrt{v} + W)}, \quad\quad &\gamma_{ik} = \delta_{ik} + \frac{v_i v_k}{2 \sqrt{v} ( 2 \sqrt{v} + W )}, \\
  h_{ij} = \frac{2}{\sqrt{v}\, W}\,\,\big( \delta_{ij} + \frac{1}{2}\, v_{ij} \big), \quad\quad
& a_{ij} = \frac{2 \sqrt{v}}{W} \gamma^{ik} \big(\delta_{kl} + \frac{1}{2}\, v_{kl} \big) \gamma^{lj}. \end{aligned}\]

Since the graph is strictly locally convex, $v$ satisfies
\begin{equation*}
\left\{ \begin{aligned}
 \Delta v + 2 n & > 0  \quad & \mbox{in} \quad \Omega_{\epsilon},\\
 v & = \epsilon^2 \quad & \mbox{on} \quad \Gamma_{\epsilon},
\end{aligned} \right.
\end{equation*}
where $\Delta$ is the Laplace-Beltrami operator in $\mathbb{R}^n$.  Let $\overline{v}$ be the solution of
\begin{equation*}
\left\{ \begin{aligned}
 \Delta \overline{v} + 2 n & = 0  \quad & \mbox{in} \quad \Omega_{\epsilon},\\
 \overline{v} & = \epsilon^2 \quad & \mbox{on} \quad \Gamma_{\epsilon}.
\end{aligned} \right.
\end{equation*}
By the comparison principle,
\[ \underline{u}^2 = \underline{v} \leq v \leq \overline{v} \quad \mbox{in} \quad \Omega_{\epsilon}. \]
Consequently,
\begin{equation} \label{eq3-8}
|D v| \leq C \quad \mbox{on} \quad \Gamma_{\epsilon},
\end{equation}
where $C$ is a positive constant depending on $\epsilon$. Hereinafter in this section, $C$ always denotes such a constant which may change from line to line. Equivalently,
\begin{equation} \label{eq3-22}
|D u| \leq C \quad \mbox{on} \quad \Gamma_{\epsilon}.
\end{equation}

For global gradient estimate, consider the test function
\[ W = \sqrt{4 v + |D v|^2}. \]
Assume its maximum is achieved at an interior point $x_0 \in \Omega_{\epsilon}$. Then at $x_0$,
\[ W W_i =  \big( v_{ki} + 2 \delta_{ki} \big) v_k = 0, \quad\quad  i = 1, \ldots, n. \]
Since the matrix $\big( v_{ki} + 2 \delta_{ki} \big)$ is positive definite,
thus $v_k = 0$ for all $k$ at $x_0$.
Along with \eqref{eqn15} and \eqref{eq3-8}, we obtain
\begin{equation} \label{eq2-1}
\max\limits_{\overline{\Omega_{\epsilon}}} |D v| \leq  \max\limits_{\overline{\Omega_{\epsilon}}} \sqrt{4 v + |D v|^2} \leq \max \Big\{ \max\limits_{\Gamma_{\epsilon}} \sqrt{4 \epsilon^2 + |D v|^2},  2 \max\limits_{\overline{\Omega_{\epsilon}}} \sqrt{v} \Big\} \leq C.
\end{equation}
Equivalently,
\begin{equation} \label{eq2-3}
\max\limits_{\overline{\Omega_{\epsilon}}} |D u|  \leq C.
\end{equation}

For second order boundary estimate, we change equation \eqref{eq1-1} under the transformation $u = \sqrt{v}$ into
\begin{equation} \label{eqn16}
G( D^2 v, \,D v, v ) = \,F ( a_{ij} ) = \,f( \lambda ( a_{ij} ) ) = \,\psi( x,\, v ).
\end{equation}

By direct calculation, we obtain the following formulae.

\begin{lemma} \label{Lemma1}
\begin{equation*}
\begin{aligned}
G^{st} = & \frac{\partial G}{\partial v_{st}} =  \frac{ \sqrt{v}}{W} F^{ij} \gamma^{is} \gamma^{t j},  \\
G_v = & \frac{\partial G}{\partial v} =   \Big(\frac{1}{2 v} - \frac{2}{W^2}\Big) F^{ij} a_{ij} + \frac{v_i v_q}{ W^2  v} F^{ij} a_{qj}, \\
G^s = & \frac{\partial G}{\partial v_s} = - \frac{v_s}{W^2} F^{ij} a_{ij} - \frac{W \gamma^{is} v_q + 2 \sqrt{v} \gamma^{qs} v_i}{\sqrt{v} W ( 2 \sqrt{v} + W )} F^{ij} a_{qj}.
\end{aligned}
\end{equation*}
In addition,
\[ \vert G^s \vert \,\leq \, C  \quad \mbox{and} \quad \vert G_v \vert \,\leq \, C. \]
\end{lemma}
\begin{proof}
Since
\[  G( D^2 v, D v, v ) =  F \Big( \frac{2 \sqrt{v}}{W} \gamma^{ik} \big(\delta_{kl} + \frac{1}{2}\, v_{kl} \big) \gamma^{lj} \Big),   \]
we have,
\begin{equation*}
 G^{st}  = \,\frac{\partial F}{\partial a_{ij}} \frac{\partial a_{ij}}{\partial v_{st}} = \frac{ \sqrt{v}}{W} F^{ij} \gamma^{is} \gamma^{t j}.
\end{equation*}

To compute $G_v$, note that
\begin{equation*}
\frac{\partial W}{\partial v} = \frac{2}{W}  \quad \mbox{and} \quad  \frac{\partial\gamma_{ik}}{\partial v} = - \frac{v_i v_k}{4 v^{3/2} W}.
\end{equation*}
Consequently,
\[ \frac{\partial\gamma^{ik}}{\partial v} = \gamma^{ip} \,\frac{v_p v_q}{4 v^{3/2} W} \,\gamma^{qk}. \]
Hence,
\[ \begin{aligned} G_v = & F^{ij} \Big( \frac{\partial}{\partial v} \big(\frac{2 \sqrt{v}}{W} \big) \gamma^{ik} (\delta_{kl} + \frac{1}{2} v_{kl} ) \gamma^{lj} + \frac{4 \sqrt{v}}{W} \frac{\partial \gamma^{ik}}{\partial v} (\delta_{kl} + \frac{1}{2} v_{kl} ) \gamma^{lj}  \Big) \\
= & \Big(\frac{1}{2 v} - \frac{2}{W^2}\Big) F^{ij} a_{ij} + \frac{\gamma^{ip} v_p v_q}{2 v^{3/2} W} F^{ij} a_{qj}.
\end{aligned} \]
We then obtain $G_v$ in view of
\[ \gamma^{ip} v_p = \,\frac{2 \sqrt{v} \,v_i}{W}. \]

For $G^s$, note that
\[ \frac{\partial W}{\partial v_s} =  \frac{v_s}{W}, \quad\quad \frac{ \partial \gamma^{ik}}{\partial v_s} = - \gamma^{ip}\, \frac{\partial \gamma_{pq}}{\partial v_s} \, \gamma^{qk}, \quad \mbox{and} \]
\begin{equation*}
\frac{\partial \gamma_{p q}}{\partial v_s} =
  \frac{\delta_{ps} v_q + \delta_{q s} v_p }{2 \sqrt{v} ( 2 \sqrt{v} + W)} - \frac{v_p v_q v_s}{2 \sqrt{v} (2 \sqrt{v} + W)^2 W}
=  \frac{\delta_{p s} v_q + v_p \gamma^{q s}}{2 \sqrt{v} ( 2 \sqrt{v} + W)}.
\end{equation*}
It follows that
\begin{equation*}
\begin{aligned}
   G^s  = & F^{ij} \Big( - \frac{ 2 \sqrt{v} v_s}{W^3} \gamma^{ik} (\delta_{kl} + \frac{1}{2} v_{kl} ) \gamma^{lj} + \frac{4 \sqrt{v}}{W} \frac{\partial \gamma^{ik}}{\partial v_s} (\delta_{kl} + \frac{1}{2} v_{kl} ) \gamma^{lj}  \Big) \\ = & - \frac{v_s}{W^2} F^{ij} a_{ij} - \frac{W \gamma^{is} v_q + 2 \sqrt{v} \gamma^{qs} v_i}{\sqrt{v} W ( 2 \sqrt{v} + W )} F^{ij} a_{qj}.
\end{aligned}
\end{equation*}
\end{proof}

For an arbitrary point on $\Gamma_{\epsilon}$, we may assume it to be the origin of $\mathbb{R}^n$. Choose a coordinate system so that the positive $x_n$ axis points to the interior normal of $\Gamma_{\epsilon}$ at the origin. There exists a uniform constant $r > 0$ such that $\Gamma_{\epsilon} \cap B_r (0)$ can be represented as a graph
\[ x_n = \rho ( x' ) = \frac{1}{2} \sum\limits_{\alpha, \beta < n} B_{\alpha \beta} x_{\alpha} x_{\beta} + O ( |x'|^3 ), \quad x' = (x_1, \ldots, x_{n - 1}).  \]

Since
\[ v = \epsilon^2 \quad \quad\mbox{on} \quad \Gamma_{\epsilon}, \]
or equivalently
\[ v ( x', \rho( x' )) = \epsilon^2, \]
we have
\begin{equation} \label{eq3-3}
 v_\alpha + v_n \,\rho_\alpha = 0
\end{equation}
and
\[ v_{\alpha \beta} +  v_{\alpha n} \rho_\beta + (v_{n \beta} + v_{n n} \rho_\beta ) \rho_\alpha + v_n \rho_{\alpha \beta} = 0. \]
Therefore,
\[ v_{\alpha \beta} (0) = - v_n (0)\, \rho_{\alpha \beta} (0), \quad\quad \alpha, \beta < n. \]
Consequently,
\begin{equation}  \label{eq2B-3}
| v_{\alpha \beta} (0) | \leq C, \quad \quad\quad  \alpha, \beta < n,
\end{equation}
where $C$ is a constant depending on $\epsilon$.

For the mixed tangential-normal derivative $v_{\alpha n} (0)$ with $\alpha < n$, note that the graph of $\underline{u}$ is strictly locally convex on $\overline{\Omega_{\epsilon}}$. Hence we have
\[ I + \frac{1}{2} \, D^2 \underline{v} \,\,\geq \, 3 \, c_0\, I \]
for some positive constant $c_0$.
Let $d(x)$ be the distance from $x \in \overline{\Omega_{\epsilon}}$ to $\Gamma_{\epsilon}$ in $\mathbb{R}^n$.
Consider the barrier function
\[ \Psi = A \, V + B\, |x|^2  \]
with
\[ V = \, v - \underline{v} + \tau  d - N d^2, \]
where the positive constant $N$, $\tau$, $B$ and $A$ are to be determined.

Define the linear operator  \,\,$ L  = \, G^{s t} \,D_{s t} +  G^s \,D_s$. By the concavity of $G$
with respect to $D^2 v$,
\begin{equation*}
\begin{aligned}
 L V   = \, & G^{st} D_{st} ( v - \underline{v} - N \,d^2 ) + \tau \,G^{st} D_{st}  d + G^s  D_s ( v - \underline{v} + \tau\, d - N \,d^2 ) \\
     \leq \, & G( D^2 v, D v, v ) - G\Big( D^2\big( \underline{v} + N \,d^2 \big) - 2 c_0 I, D v, v \Big)
       \\ & + ( C \tau - 2 c_0 ) \sum G^{ii} + C  ( 1 +  \tau +  N \delta ).
\end{aligned}
\end{equation*}
Note that
\begin{equation*}
    I + \frac{1}{2} \,D^2\big( \underline{v} + N \,d^2 \big ) -  c_0 I
\geq \,\,  2 c_0 I + N D d \otimes D d - C N \delta I := \mathcal{H}.
\end{equation*}
Denote $\gamma = ( \gamma^{ik} )$. We have
\begin{equation*}
\begin{aligned}
& G\Big( D^2\big( \underline{v} + N \,d^2 \big) - 2 c_0 I, D v, v \Big)  = \,  F \Big( \frac{2 \sqrt{v}}{W} \gamma \big( I + \frac{1}{2} \,D^2\big( \underline{v} + N \,d^2 \big ) -  c_0 I  \big) \gamma  \Big) \\
\geq  &\,\,F \Big( \frac{2 \sqrt{v}}{W} \gamma \,\mathcal{H}\, \gamma \Big)
=  \,F \Big( \frac{2 \sqrt{v}}{W} \,\mathcal{H}^{1/2} \,\gamma\gamma \,\mathcal{H}^{1/2} \Big)
\geq  \, F ( \tilde{c}\, \mathcal{H} ),
\end{aligned}
\end{equation*}
where $\tilde{c}$ is a positive constant.
Hence
\begin{equation*}
L V \leq \,  - F ( \tilde{c}\, \mathcal{H} )
       + ( C \tau - 2 c_0 ) \sum G^{ii} + C  ( 1 +  \tau +  N \delta ).
\end{equation*}
Note that $\mathcal{H} = \mbox{diag} \Big( 2 c_0 - C N \delta,\,\, \ldots,\,\, 2 c_0 - C N \delta,\,\, 2 c_0 - C N \delta + N\Big)$.
We can choose $N$ sufficiently large and $\tau$, $\delta$ sufficiently small ($\delta$ depends on $N$) such that
\[  C \tau  \leq c_0, \quad  C N \delta \leq c_0, \quad  -F ( \tilde{c}\,\mathcal{H} )   + C  + 2 c_0  \leq - 1. \]
Hence the above inequality becomes
\begin{equation} \label{eq3-1}
 L V \leq - c_0 \sum G^{ii} - 1.
\end{equation}
We then require $\delta \leq \frac{\tau}{N}$ so that
\[ V \,\geq \, 0 \quad \mbox{in} \quad \Omega_{\epsilon} \cap B_{\delta} ( 0 ). \]
By Lemma \ref{Lemma1},
\begin{equation*}
L \big(  |x|^2 \big)
\leq  \,C \big(1 + \sum G^{ii}\big).
\end{equation*}
This, together with \eqref{eq3-1} yields,
\begin{equation} \label{eq3-2}
L \Psi \,\leq \,A \big( - c_0 \sum G^{ii} - 1 \big) + B C \big( 1 + \sum G^{ii} \big) \quad\mbox{in}\quad \Omega_{\epsilon} \cap B_{\delta} ( 0 ).
\end{equation}

Now, we consider the operator
\[ T = \partial_\alpha + \sum\limits_{\beta < n} B_{\alpha \beta} ( x_{\beta} \partial_n - x_n \partial_{\beta} ).  \]
Note that for $\delta > 0$ sufficiently small,
\[ \vert T v   \vert \leq \,C   \,\, \quad \quad \mbox{in} \quad \Omega_{\epsilon} \cap B_{\delta} ( 0 ).  \]
Also, in view of \eqref{eq3-3},
\[ \vert T v  \vert \leq C \,|x|^2  \,\, \quad \quad \mbox{on} \quad \Gamma_{\epsilon} \cap B_{\delta} ( 0 ).  \]

To compute $L( T v )$, we need the following lemma (see \cite{GSS09}).

\begin{lemma} \label{Lemma2}
For $1 \leq i, j \leq n$,
\[ ( L + G_v - \psi_v )( x_i v_j - x_j v_i ) = x_i \psi_{x_j} - x_j \psi_{x_i}. \]
\end{lemma}
\begin{proof}
For $\theta \in \mathbb{R}$, let
\begin{equation*}
\begin{aligned}
 y_i = \, & x_i \cos\theta - x_j \sin\theta, \\
 y_j = \, & x_i \sin\theta + x_j \cos\theta, \\
 y_k = \, & x_k, \quad k \neq i, j.
\end{aligned}
\end{equation*}
Since $G - \psi$ is invariant for the rotations of $\mathbb{R}^n$, we have
\[ G( D^2 v( y ), D v( y ), v( y ) ) = \psi( y, v( y ) ). \]
Differentiate with respect to $\theta$ and change the order of differentiation,
\[ ( L  +  G_v - \psi_v) \vert_y \,\frac{\partial v}{\partial \theta} = \psi_{y_i} \frac{\partial y_i}{\partial \theta} + \psi_{y_j} \frac{\partial y_j}{\partial \theta}. \]
Set $\theta = 0$ in the above equality and notice that at $\theta = 0$,
\[ y = x, \quad\quad \frac{\partial y_i}{\partial \theta} = - x_j, \quad\quad \frac{\partial y_j}{\partial \theta} = x_i, \quad\quad \frac{\partial v}{\partial \theta} = x_i v_j - x_j v_i. \]
We thus proved the lemma.
\end{proof}

By Lemma \ref{Lemma2} and Lemma \ref{Lemma1}, we have
\begin{equation} \label{eq3-4}
\vert L ( T v ) \vert \leq C.
\end{equation}
Choose $B$ sufficiently large such that
\[ \Psi \pm  T v \geq 0 \quad \mbox{on} \quad \partial(\Omega_{\epsilon} \cap B_{\delta} ( 0 )).   \]
From  \eqref{eq3-2} and \eqref{eq3-4} we have
\[ L ( \Psi \pm T v ) \leq A \big( - c_0 \sum G^{ii} - 1 \big) + B C \big( 1 + \sum G^{ii} \big) + C. \]
Choose $A$ sufficiently large such that
\[ L ( \Psi \pm T v ) \leq 0 \quad \mbox{in} \quad \Omega_{\epsilon} \cap B_{\delta} ( 0 ). \]
By the maximum principle,
\[ \Psi \pm T v \geq 0  \quad \mbox{in} \quad \Omega_{\epsilon} \cap B_{\delta} ( 0 ), \]
which implies
\begin{equation} \label{eq3-5}
\vert v_{\alpha n} (0) \vert \leq C.
\end{equation}

Up to now, we have proved that
\[ |v_{\xi\eta} (x)| \leq C, \quad | v_{\xi \gamma} (x) | \leq C, \quad\quad \forall \quad x \in \Gamma_{\epsilon}, \]
where $\xi$ and $\eta$ are any unit tangential vectors and $\gamma$ the unit interior normal vector to $\Gamma_{\epsilon}$ on $\Omega_{\epsilon}$.
It suffices to give an upper bound
\begin{equation} \label{eq3-10}
v_{\gamma\gamma}  \leq C  \quad \mbox{on} \quad \Gamma_{\epsilon}.
\end{equation}
Motivated by \cite{Cruz} (see also \cite{Guan99, Tru}), we derive \eqref{eq3-10}.

First recall some general facts. The projection of $\Gamma_k \subset \mathbb{R}^n$ onto $\mathbb{R}^{n - 1}$ is exactly
\[
\Gamma'_{k - 1} = \,\{ (\lambda_1, \ldots, \lambda_{n-1}) \in \mathbb{R}^{n - 1} \,| \, \sigma_{j} ( \lambda_1, \ldots, \lambda_{n-1} ) > 0,\,\, \, j = 1, \ldots, k - 1 \}.
\]
Let $\kappa' = (\kappa'_1, \ldots, \kappa'_{n-1})$ be the roots of
\begin{equation}  \label{eq3-11}
\det (  \kappa'_{\zeta} \, g_{\alpha\beta} - h_{\alpha \beta}  ) = 0,
\end{equation}
where $(h_{\alpha\beta})$ and $(g_{\alpha\beta})$ are the first $(n - 1) \times (n - 1)$ principal minors of  $(h_{ij})$ and $(g_{ij})$ respectively.
Then $\kappa [v] \in \Gamma_k$ implies $\kappa'[v] \in \Gamma'_{k - 1}$, and this is true for any local frame field.  Note that $\kappa'[v]$ may not be $(\kappa_1, \ldots, \kappa_{n-1})[v]$.

For $x \in \Gamma_{\epsilon}$, let the indices in \eqref{eq3-11} be given by the tangential directions to $\Gamma_{\epsilon}$ and $\kappa'[v](x)$ be the roots of \eqref{eq3-11}.   Define
\[ \tilde{d} (x)  =\, \sqrt{v} \,W  \,\,\mbox{dist} ( \kappa'[v](x), \,\partial \Gamma'_{k - 1} ) \quad\quad \mbox{and} \quad\quad m = \min\limits_{x \in \Gamma_{\epsilon}}\,\tilde{d} (x). \]
Choose a coordinate system in $\mathbb{R}^n$ such that $m$ is achieved at $0 \in \Gamma_{\epsilon}$  and the positive $x_n$ axis points to the interior normal of $\Gamma_{\epsilon}$ at $0$.
We want to prove that $m$ has a uniform positive lower bound.

Let $\xi_1, \ldots, \xi_{n-1}, \gamma$ be a local frame field around $0$ on $\Omega_{\epsilon}$, obtained by parallel translation of a local frame field $\xi_1, \ldots, \xi_{n-1}$ around $0$ on $\Gamma_{\epsilon}$ satisfying
\[ g_{\alpha \beta} = \delta_{\alpha\beta}, \quad\quad h_{\alpha \beta}(0) = \kappa'_{\alpha}(0) \, \delta_{\alpha\beta}, \quad\quad \kappa'_1 (0) \leq \ldots \leq \kappa'_{n-1} (0) \]
and the interior, unit, normal vector field $\gamma$ to $\Gamma_{\epsilon}$, along the directions perpendicular to $\Gamma_{\epsilon}$ on $\Omega_{\epsilon}$.
We can see that this choice of frame field has nothing to do with $v$ (or equivalently, $u$). In fact, if we denote
\[ \xi_{\alpha} = \sum_{\beta = 1}^{n - 1} \eta_{\alpha}^{\beta} \,e_{\beta}, \quad\quad \alpha = 1, \ldots, n - 1,  \]
where $e_1, \ldots, e_{n - 1}$ is a fixed local orthonormal frame on $\Gamma_{\epsilon}$,
and consider a general boundary value condition, say $v = \varphi$ on $\Gamma_{\epsilon}$,
then on $\Gamma_{\epsilon}$,
\[ \begin{aligned}
g_{\alpha\beta} = & \frac{1}{u^2} \Big( \xi_{\alpha} \cdot \xi_{\beta} + D_{\xi_{\alpha}} u \, D_{\xi_{\beta}} u \Big) = \frac{1}{\varphi} \Big( \xi_{\alpha} \cdot \xi_{\beta} + D_{\xi_{\alpha}} (\sqrt{\varphi}) \, D_{\xi_{\beta}} (\sqrt{\varphi})  \Big)
\\= & \frac{1}{\varphi} \sum\limits_{\tau, \zeta = 1}^{n - 1} \eta_{\alpha}^{\tau} \,\Big( \delta_{\tau\zeta} + \frac{D_{e_{\tau}}\varphi \, D_{e_{\zeta}}\varphi}{4 \varphi} \Big) \,\eta_{\beta}^{\zeta}.
\end{aligned} \]
Note that there exist $\eta_{\alpha}^{\tau}$
for $\alpha, \tau  = 1, \ldots, n-1$ such that $g_{\alpha\beta} = \delta_{\alpha\beta}$ on $\Gamma_{\epsilon}$. By a rotation, we can further make $(h_{\alpha \beta}(0))$ to be diagonal.

By Lemma 6.1 of \cite{CNSIII}, there exists $\mu = (\mu_1, \ldots, \mu_{n-1}) \in \mathbb{R}^{n - 1}$ with $\mu_1 \geq \ldots \geq \mu_{n - 1} \geq 0$ such that \[ \sum\limits_{\alpha = 1}^{n - 1} \mu_{\alpha}^2 = 1, \quad \quad \Gamma'_{k-1} \subset \{ \lambda' \in \mathbb{R}^{n-1} \,|\, \mu \cdot \lambda' > 0 \} \quad \quad \mbox{and}   \]
\begin{equation} \label{eq3-35}
m = \tilde{d}(0) =   \,\sqrt{v} \, W \, \sum\limits_{\alpha < n} \mu_{\alpha} \,\kappa'_{\alpha} (0) =  \,\sum\limits_{\alpha < n} \,\mu_{\alpha} \,\big( D_{\xi_{\alpha} \xi_{\alpha}} v  +  2 \,\xi_{\alpha} \cdot \xi_{\alpha} \big) (0).
\end{equation}

Since $\underline{v}$ is strictly locally convex near $\Gamma_{\epsilon}$ and $\sum \mu_{\alpha} \geq 1$,
\[\sum\limits_{\alpha < n}  \mu_{\alpha} \big( D_{\xi_{\alpha} \xi_{\alpha}}  \underline{v} + 2 \,\xi_{\alpha} \cdot \xi_{\alpha} \big) (0)  \,\geq \,  2 \, c_1  \]
for a uniform positive constant $c_1$.  Consequently,
\begin{equation} \label{eq3-6}
\begin{aligned}
&  (\underline{v} - v)_{\gamma} (0)\, \sum\limits_{\alpha < n} \mu_{\alpha} \, d_{\xi_{\alpha} \xi_{\alpha}} (0)  = \sum\limits_{\alpha < n} \mu_{\alpha} D_{\xi_{\alpha} \xi_{\alpha}} ( \underline{v} - v ) (0) \\ = \,\, & \sum\limits_{\alpha < n}  \mu_{\alpha} \big( D_{\xi_{\alpha} \xi_{\alpha}}  \underline{v} + 2 \,\xi_{\alpha} \cdot \xi_{\alpha} \big) (0) - \sum\limits_{\alpha < n}  \mu_{\alpha} \big( D_{\xi_{\alpha} \xi_{\alpha}}  v +  2 \,\xi_{\alpha} \cdot \xi_{\alpha} \big) (0) \geq  2 \, c_1 - \tilde{d}(0).
\end{aligned}
\end{equation}
The first line in \eqref{eq3-6} is true, since we can write $v - \underline{v} = \omega \,\,d$
for some function $\omega$ defined in a neighborhood of $\Gamma_{\epsilon}$ in $\Omega_{\epsilon}$.
Differentiate this identity,
\[ ( v - \underline{v} )_i = \,\omega_i \,\,d + \omega \,d_i, \quad \quad ( v - \underline{v} )_{\gamma} = \,\omega_{\gamma} \,d + \omega \,d_{\gamma},\]
\[ ( v - \underline{v} )_{ij} = \,\omega_{ij} \,\,d + \omega_i \,d_j + \omega_j \,d_i + \omega \, d_{ij}. \]
Note that $d_{\xi_{\alpha}} (0) = 0$ and $d_{\gamma} (0) = 1$. Thus,
\begin{equation*}
D_{\xi_{\alpha} \xi_{\alpha}}( v - \underline{v} ) (0)  =  ( v - \underline{v} )_{\gamma}(0) \, d_{\xi_{\alpha} \xi_{\alpha}}(0).
\end{equation*}

We may assume
$\tilde{d} (0) \leq  \, c_1$, for, otherwise we are done. Then from \eqref{eq3-6},
\[  ( \underline{v} - v )_{\gamma} (0) \sum\limits_{\alpha < n} \mu_{\alpha}\, d_{\xi_{\alpha} \xi_{\alpha}}  (0) \geq    c_1. \]
Since $0 <  ( v - \underline{v} )_{\gamma} (0) \leq C$,
\[ \sum\limits_{\alpha < n} \mu_{\alpha}\, d_{\xi_{\alpha} \xi_{\alpha}} (0)  \leq  - \,2 \, c_2  \]
for some uniform constant $c_2 > 0$.
By continuity of $d_{\xi_{\alpha} \xi_{\alpha}} (x)$ at $0$ and $0 \leq \mu_{\alpha} \leq 1$,
\begin{equation*}
 \sum\limits_{\alpha < n} \mu_{\alpha}\,\Big( d_{\xi_{\alpha} \xi_{\alpha}} (x)  - d_{\xi_{\alpha} \xi_{\alpha}} (0) \Big)  <  \sum\limits_{\alpha < n} \mu_{\alpha}\,\frac{c_2}{n - 1}    \leq c_2  \quad \quad\mbox{in} \quad
 \Omega_\epsilon \cap B_{\delta}(0)
\end{equation*}
for some uniform constant $\delta > 0$. Thus
\begin{equation} \label{eq3-7}
\sum\limits_{\alpha < n} \mu_{\alpha} \,d_{\xi_{\alpha} \xi_{\alpha}} (x) \,<\, - c_2  \quad \quad\mbox{in} \quad
 \Omega_\epsilon \cap B_{\delta}(0).
\end{equation}

On the other hand, by Lemma 6.2 of \cite{CNSIII}, for any $x \in \Gamma_{\epsilon}$ near $0$,
\[\begin{aligned}
& \,\sum\limits_{\alpha < n} \,\mu_{\alpha} \,\Big( D_{\xi_{\alpha} \xi_{\alpha}} v + 2 \,\xi_{\alpha} \cdot \xi_{\alpha} \Big) (x) \,=  \sum\limits_{\alpha < n} \,\mu_{\alpha}  \sqrt{v}\, W \,h_{\alpha \alpha} (x)  \\
\geq & \, \sqrt{v} \, W \,\sum\limits_{\alpha < n} \,\mu_{\alpha} \,\kappa'_{\alpha} [ v ] (x) \,\geq \,\tilde{d} (x)\, \geq\, \tilde{d} (0).
\end{aligned}\]
Thus for any $x \in \Gamma_{\epsilon}$ near $0$,
\begin{equation}  \label{eq3-34}
\begin{aligned}
& ( v - \varphi)_{\gamma} (x) \sum\limits_{\alpha < n} \mu_{\alpha}\, d_{\xi_{\alpha} \xi_{\alpha}} (x) = \sum\limits_{\alpha < n} \mu_{\alpha}\, D_{\xi_{\alpha} \xi_{\alpha}} ( v - \varphi ) (x) \\
= \,\, & \sum\limits_{\alpha < n}  \mu_{\alpha} \Big( D_{\xi_{\alpha} \xi_{\alpha}}  v + 2\, \xi_{\alpha} \cdot \xi_{\alpha} \Big) (x) - \sum\limits_{\alpha < n}  \mu_{\alpha} \Big( D_{\xi_{\alpha} \xi_{\alpha}}  \varphi + 2 \,\xi_{\alpha} \cdot \xi_{\alpha} \Big) (x)  \\
\geq  \,\,&\, \tilde{d}(0) - \sum\limits_{\alpha < n}  \mu_{\alpha} \Big( D_{\xi_{\alpha} \xi_{\alpha}}  \varphi +  2 \,\xi_{\alpha} \cdot \xi_{\alpha} \Big) (x).
\end{aligned}
\end{equation}
In view of \eqref{eq3-7}, define in $\Omega_\epsilon \cap B_{\delta}(0)$,
\[ \Phi \, =  \,  \frac{1}{\sum\limits_{\alpha < n} \mu_{\alpha}\, d_{\xi_{\alpha} \xi_{\alpha}} } \,\left( \,\tilde{d}(0) - \sum\limits_{\alpha < n}  \mu_{\alpha} \Big( D_{\xi_{\alpha} \xi_{\alpha}}  \varphi + 2 \,\xi_{\alpha} \cdot \xi_{\alpha} \Big)  \right)   -  ( v - \varphi)_{\gamma}. \]
By \eqref{eq3-7} and \eqref{eq3-34},  $\Phi \geq 0$ on $\Gamma_{\epsilon} \cap B_{\delta} (0)$.
In addition, we have in $\Omega_\epsilon \cap B_{\delta}(0)$,
\begin{equation} \label{eq3-9}
L(\Phi) \leq \, C \big( 1 + \sum G^{ii} \big) - L \Big( D ( v - \varphi ) \cdot D d \Big)
\leq  \, C \big( 1 + \sum G^{ii} \big).
\end{equation}
This is because $0 \leq \mu_{\alpha} \leq 1$ and
\[ \begin{aligned}
& \Big| L \big( D ( v  - \varphi ) \cdot D d \big)  \Big|  =   \Big| D d \cdot L \big( D (v - \varphi) \big) + D ( v - \varphi) \cdot L (D d) + 2 G^{st} (v - \varphi)_{is} d_{it} \Big| \\
\leq & \, C \big( 1 +  \sum G^{ii} \big) \, +  \,\Big\vert 2\, G^{st}\, d_{it} \Big( \frac{W}{\sqrt{v}} \gamma_{ki} \gamma_{sl} a_{kl} - 2 \delta_{is} \Big) \Big\vert \\
= & \, C \big( 1 +  \sum G^{ii} \big) \, +  \,\Big\vert 2\,\gamma_{ki} d_{it} \gamma^{tj}\, F^{lj}\, a_{kl} \, - 4 \, G^{st}\, d_{st} \Big\vert
\leq  \, C \big( 1 +  \sum G^{ii} \big).
 \end{aligned} \]
By \eqref{eq3-2} and \eqref{eq3-9}, we may choose $A > > B > > 1$  such that $\Psi + \Phi \geq 0$ on $\partial(\Omega_{\epsilon} \cap B_{\delta}(0))$ and $L(\Psi + \Phi) \leq 0$ in $\Omega_{\epsilon} \cap B_{\delta}(0)$. By the maximum principle, $\Psi + \Phi \geq 0$ in $\Omega_{\epsilon} \cap B_{\delta}(0)$. Since $(\Psi + \Phi)(0) = 0$ by \eqref{eq3-34} and \eqref{eq3-35}, we have $(\Psi + \Phi)_n (0) \geq 0$. Therefore, $v_{nn} (0) \leq C$, which,
together with \eqref{eq2B-3} and \eqref{eq3-5}, gives a bound $|D^2 v (0)| \leq C$, and consequently a bound for all the principal curvatures at $0$. By \eqref{eqn5},
\[ \mbox{dist} ( \kappa[v](0), \,\partial\Gamma_k )  \geq c_3  \]
and therefore on $\Gamma_{\epsilon}$,
\[ \tilde{d}(x) \geq \tilde{d}(0) = \sqrt{v}\, W\,\mbox{dist} ( \kappa'[v](0), \,\partial\Gamma'_{k - 1} )   \geq c_4,  \]
where $c_3$ and $c_4$ are positive uniform constants.

By a proof similar to Lemma 1.2 of \cite{CNSIII}, we know that there exists $R > 0$ depending on the bounds \eqref{eq2B-3} and \eqref{eq3-5} such that if $v_{\gamma\gamma}(x_0) \geq R$ and $x_0 \in \Gamma_{\epsilon}$, then the principal curvatures $(\kappa_1, \ldots, \kappa_n)$ at $x_0$ satisfy
\[ \kappa_{\alpha} = \kappa'_{\alpha} + o(1), \quad\quad \alpha < n, \]
\[ \kappa_n = \frac{ h_{nn} - g_{1n} h_{n1} - \ldots - g_{n n-1} h_{n n-1} }{g_{n n} - g_{1 n}^2 - \ldots - g_{n n-1}^2} \Big( 1 + \mathcal{O} \Big( \frac{g_{n n} - g_{1 n}^2 - \ldots - g_{n n-1}^2}{ h_{nn} - g_{1n} h_{n1} - \ldots - g_{n n-1} h_{n n-1} } \Big) \Big) \]
in the local frame $\xi_1, \ldots, \xi_{n-1}, \gamma$ around $x_0$.
When $R$ is sufficiently large, we have
\[ G( D^2 v, D v, v ) (x_0) \, > \, \psi(x_0, \epsilon^2),\]
contradicting with equation \eqref{eqn16}.
Hence $v_{\gamma\gamma} <  R$ on $\Gamma_{\epsilon}$. \eqref{eq3-10} is proved.

\vspace{6mm}

\section{Global curvature estimates}

\vspace{4mm}

For a hypersurface $\Sigma\subset\mathbb{H}^{n+1}$, let $g$ and $\nabla$ be the induced hyperbolic metric and Levi-Civita connection on $\Sigma$ respectively, and let $\tilde{g}$ and $\tilde{\nabla}$ be the metric and Levi-Civita connection induced from $\mathbb{R}^{n+1}$ when $\Sigma$ is viewed as a hypersurface in $\mathbb{R}^{n+1}$. The Christoffel symbols associated with $\nabla$ and $\tilde{\nabla}$ are related by the formula
\[ \Gamma_{ij}^k = \tilde{\Gamma}_{ij}^k - \frac{1}{u} (u_i \delta_{kj} + u_j \delta_{ik} - \tilde{g}^{kl} u_l \tilde{g}_{ij}). \]
Consequently, for any $v \in C^2(\Sigma)$,
\begin{equation} \label{eqC-3}
\nabla_{ij} v = (v_i)_j - \Gamma_{ij}^k v_k = \tilde{\nabla}_{ij} v + \frac{1}{u}( u_i v_j + u_j v_i - \tilde{g}^{kl} u_l v_k \tilde{g}_{ij} ).
\end{equation}
Note that \eqref{eqC-3} holds for any local frame.

\begin{lemma}  \label{Lemma0-3}
In $\mathbb{R}^{n+1}$, we have the following identities.
\begin{equation} \label{eq0-1}
\tilde{g}^{kl} u_k u_l  = |\tilde\nabla u|^2 = 1 - (\nu^{n+1})^2,
\end{equation}
\begin{equation}  \label{eq0-2}
\tilde{\nabla}_{ij} u = \tilde{h}_{ij} \nu^{n+1} \quad \mbox{and} \quad \tilde{\nabla}_{ij} x_{k} = \tilde{h}_{ij} \nu^{k}, \quad k = 1, \ldots, n,
\end{equation}
\begin{equation}  \label{eq0-4}
(\nu^{n+1})_i = - \tilde{h}_{ij} \,\tilde{g}^{j k} u_k,
\end{equation}
\begin{equation} \label{eq0-5}
\tilde{\nabla}_{ij} \nu^{n+1} = - \tilde{g}^{kl} ( \nu^{n+1} \tilde{h}_{il} \tilde{h}_{kj} + u_l \tilde{\nabla}_k \tilde{h}_{ij} ),
\end{equation}
where $\tau_1, \ldots, \tau_n$ is any local frame on $\Sigma$.
\end{lemma}
\begin{proof}
To prove \eqref{eq0-1}, we may write
\begin{equation} \label{eq0-8}
\partial_{n + 1} = \sum\limits_{k = 1}^n a_k \tau_k + b \nu.
\end{equation}
Taking inner product of \eqref{eq0-8} with $\nu$ in $\mathbb{R}^{n+1}$,  we obain
\[ \nu^{n + 1} = \partial_{n + 1} \cdot \nu = b.\]
Taking inner product of \eqref{eq0-8} with $\tau_j$ in $\mathbb{R}^{n+1}$,  we have
\[ u_j = (X \cdot \partial_{n+1})_j = \partial_{n + 1} \cdot \tau_j = a_k \tau_k \cdot \tau_j = a_k \tilde{g}_{kj},  \]
where $X$ is the position vector field of $\Sigma$ (note that this is different from the conformal Killing field when using half space model for $\mathbb{H}^{n + 1}$).
Thus,
\[ a_k = u_j \tilde{g}^{jk}. \]
Therefore,
\[ \partial_{n + 1} =  u_j \tilde{g}^{jk} \tau_k +  \nu^{n + 1} \nu =  \tilde{\nabla} u +  \nu^{n + 1} \nu, \]
which implies \eqref{eq0-1}.

For \eqref{eq0-2}, note that
\[\begin{aligned}
\, & \tilde{\nabla}_{ij} (X \cdot \partial_k) = \big( (X \cdot \partial_k)_j \big)_i - \tilde{\Gamma}_{ij}^l (X \cdot \partial_k)_l \\
= \, &  (\tau_j \cdot \partial_k)_i - \tilde{\Gamma}_{ij}^l \,\tau_l \cdot \partial_k = \tilde{D}_{\tau_i} \tau_j \cdot \partial_k - \tilde{\Gamma}_{ij}^l \,\tau_l \cdot \partial_k \\
= \, & ( \tilde{\nabla}_{\tau_i} \tau_j +  \tilde{h}_{ij} \nu ) \cdot \partial_k - \tilde{\Gamma}_{ij}^l \,\tau_l \cdot \partial_k = \tilde{h}_{ij} \nu  \cdot \partial_k, \quad \quad k = 1, \ldots, n+1.
\end{aligned}\]
Here we have applied the Gauss formula for $\Sigma$ as a hypersurface in $\mathbb{R}^{n+1}$.

For \eqref{eq0-4}, by the Weingarten formula for $\Sigma$ as a hypersurface in $\mathbb{R}^{n+1}$, we have
\[ (\nu^{n+1})_i = (\nu \cdot \partial_{n + 1})_i  = \tilde{D}_{\tau_i} \nu \cdot \partial_{n + 1} = - \tilde{h}_{ik} \,\tilde{g}^{kl} \tau_l \cdot \partial_{n + 1} = - \tilde{h}_{ik} \tilde{g}^{kl} u_l. \]

Finally, \eqref{eq0-5} follows from \eqref{eq0-4}, \eqref{eq0-2} and the Codazzi equation for $\Sigma$ as a hypersurface in $\mathbb{R}^{n + 1}$. In fact,
\[\tilde{\nabla}_{ij} \nu^{n+1} = - \tilde{g}^{kl} ( u_l \tilde{\nabla}_i \tilde{h}_{jk}  + \tilde{h}_{jk}  \tilde{\nabla}_{il} u ) = - \tilde{g}^{kl} (  u_l \tilde{\nabla}_k \tilde{h}_{ij} + \nu^{n+1} \tilde{h}_{il} \tilde{h}_{jk} ). \]
\end{proof}

\begin{lemma}  \label{LemmaC-1}
Let $\Sigma$ be a strictly locally convex hypersurface in $\mathbb{H}^{n+1}$ satisfying equation \eqref{eq1-1}. Then in a local orthonormal frame on $\Sigma$,
\begin{equation}  \label{eqC-5}
\begin{aligned}
F^{ij} \nabla_{ij} \nu^{n+1}
= & - \nu^{n+1} F^{ij} h_{ik} h_{kj} + \big( 1 + (\nu^{n+1})^2 \big) F^{ij} h_{ij} - \nu^{n+1} \sum f_i \\ & - \frac{2}{u^2} F^{ij} h_{jk} u_i u_k + \frac{2 \nu^{n+1}}{u^2} F^{ij} u_i u_j - \frac{u_k}{u} \psi_k.
\end{aligned}
\end{equation}
\end{lemma}
\begin{proof}
By \eqref{eqC-3}, \eqref{eq0-5},
\begin{equation} \label{eq0-6}
\begin{aligned}
\,\,& F^{ij} \nabla_{ij} \nu^{n+1} \\ = & \, \,F^{ij} \Big(\tilde{\nabla}_{ij} \nu^{n+1} + \frac{1}{u}\big( u_i (\nu^{n+1})_j + u_j (\nu^{n+1})_i - \tilde{g}^{kl} u_l (\nu^{n+1})_k \tilde{g}_{ij} \big)\Big) \\
= & \,\, - \frac{\nu^{n+1}}{u^2}  F^{ij} \tilde{h}_{ik} \tilde{h}_{kj} - \frac{u_k}{u^2} F^{ij} \tilde{\nabla}_k \tilde{h}_{ij}  - \frac{2}{u^3} F^{ij} \tilde{h}_{jk} u_i u_k - \frac{u_k}{u} (\nu^{n+1})_k \,\sum f_i.
\end{aligned}
\end{equation}
Since $\Sigma$ can also be viewed as a hypersurface in $\mathbb{R}^{n+1}$,
\begin{equation*}
F(g^{il} h_{lj}) = F\Big( u^2 \tilde{g}^{il} \big( \frac{1}{u}\,\tilde{h}_{lj} + \frac{\nu^{n+1}}{u^2}\, \tilde{g}_{lj} \big)\Big) = F \Big(  u \, \tilde{g}^{il}\,\tilde{h}_{lj}  +  \nu^{n+1} \delta_{ij} \Big) = \psi.
\end{equation*}
Differentiate this equation with respect to $\tilde{\nabla}_k$ and then multiply by $\frac{u_k}{u}$,
\[ \frac{u_k^2}{u^3}\, F^{ij} \tilde{h}_{ij} + \frac{u_k}{u^2} F^{ij} \tilde{\nabla}_k \tilde{h}_{ij} + \frac{u_k}{u} (\nu^{n+1})_k \sum f_i = \, \frac{ u_k}{u} \,\psi_k.  \]
Take this identity into \eqref{eq0-6},
\begin{equation*}
F^{ij} \nabla_{ij} \nu^{n+1} =  - \frac{\nu^{n+1}}{u^2} F^{ij} \tilde{h}_{ik} \tilde{h}_{kj} - \frac{2}{u^3} F^{ij} \tilde{h}_{jk} u_i u_k + \frac{u_k^2}{u^3} F^{ij} \tilde{h}_{ij} - \frac{u_k}{u} \,\psi_k.
\end{equation*}
In view of \eqref{eq0-3}, we obtain \eqref{eqC-5}.
\end{proof}

For global curvature estimates, we use the method in \cite{GS11}. Assume
\[ \nu^{n+1} \,\geq \,2\, a > 0 \quad \quad \mbox{on} \quad \Sigma \]
for some constant $a$.  Let $\kappa_{\max} ({ \bf x })$ be the largest principal curvature of $\Sigma$ at ${\bf x}$. Consider
\begin{equation*}
M_0 = \sup\limits_{{\bf x} \in\Sigma}
\,\frac{\kappa_{\max\,}({\bf x})}{{\nu}^{n+1} - a}.
\end{equation*}
Assume $M_0 > 0$ is attained at an interior point ${ \bf x}_0 \in \Sigma$.
Let $\tau_1, \ldots, \tau_n$ be a local orthonormal frame about
${ \bf x}_0$ such that $h_{ij}({\bf x}_0) = \kappa_i \,\delta_{ij}$, where
$\kappa_1, \ldots, \kappa_n$ are the hyperbolic principal curvatures of
$\Sigma$ at ${\bf x}_0$. We may assume $\kappa_1 = \kappa_{\max\,}({\bf x}_0)$.
Thus, $\ln h_{11} - \ln ( {\nu}^{n+1} - a )$ has a
local maximum at ${\bf x}_0$, at which,
\begin{equation} \label{eq2G-1}
\frac{h_{11i}}{h_{11}} - \frac{\nabla_i \nu^{n + 1}}{\nu^{ n + 1 } - a}  = 0,
\end{equation}

\begin{equation} \label{eq2G-2}
\frac{h_{11ii}}{h_{11}} - \frac{\nabla_{ii} \nu^{n + 1}}{\nu^{n + 1} - a} \,\leq 0.
\end{equation}

Differentiate equation \eqref{eq1-1} twice,
\begin{equation}  \label{eq2G-3}
F^{ii}\,h_{ii11}\, + \,F^{ij,\,rs} h_{ij1} h_{rs1}\,=\,\psi_{11} \, \geq \, - C \kappa_1.
\end{equation}

By Gauss
equation, we have the following formula when changing the order of
differentiation for the second fundamental form,
\begin{equation} \label{eq2G-4}
h_{iijj} = h_{jjii} + ( \kappa_i\,\kappa_j - 1 )\,( \kappa_i -
\kappa_j ).
\end{equation}

Combining \eqref{eq2G-2}, \eqref{eq2G-3}, \eqref{eq2G-4} and \eqref{eqC-5} yields,

\begin{equation} \label{eq2G-5}
\begin{aligned}
& \Big( \kappa_1^2 - \frac{1 + (\nu^{n+1})^2}{\nu^{n+1} - a} \kappa_1 + 1 \Big) \,\sum f_i\,\kappa_i + \frac{a \kappa_1}{\nu^{n+1} - a} \big(\sum f_i + \sum f_i \, \kappa_i^2 \big) \\
& - F^{ij, rs}\,h_{ij1}\,h_{rs1}
+ \frac{2 \kappa_1}{\nu^{n+1} - a} \, \sum f_i \frac{u_i^2}{u^2} \big(\kappa_i - \nu^{n+1}\big) - C \kappa_1 \, \quad \leq 0.
\end{aligned}
\end{equation}

Next, take \eqref{eq0-4}, \eqref{eq0-3} into \eqref{eq2G-1},
\begin{equation*}
 h_{11i} = \frac{\kappa_1}{\nu^{n+1} - a}\, \frac{u_i}{u} (\nu^{n+1} - \kappa_i),
\end{equation*}
and recall an inequality of Andrews \cite{And} and Gerhardt \cite{Ger},
\[  - F^{ij, rs}\,h_{ij1}\,h_{rs1} \, \geq \, \sum\limits_{i \neq j} \frac{f_i - f_j}{\kappa_j - \kappa_i} h_{ij1}^2 \,\geq \, 2 \sum\limits_{i \geq 2} \frac{f_i - f_1}{\kappa_1 - \kappa_i}\, h_{i11}^2. \]
Therefore, \eqref{eq2G-5} becomes,

\begin{equation} \label{eq2G-6}
\begin{aligned}
& 0  \geq  \Big( \kappa_1^2 - \frac{1 + (\nu^{n+1})^2}{\nu^{n+1} - a} \kappa_1 + 1 \Big) \,\sum f_i\,\kappa_i - C \kappa_1 + \frac{a \kappa_1}{\nu^{n+1} - a} \big(\sum f_i + \sum f_i \, \kappa_i^2 \big) \\
&  + \frac{2 \,\kappa_1^2}{(\nu^{n+1} - a)^2}\,\sum\limits_{i \geq 2} \frac{f_i - f_1}{\kappa_1 - \kappa_i}\, \frac{u_i^2}{u^2} (\nu^{n+1} - \kappa_i)^2
+ \frac{2 \kappa_1}{\nu^{n+1} - a} \, \sum f_i \frac{u_i^2}{u^2} \big(\kappa_i - \nu^{n+1}\big).
\end{aligned}
\end{equation}
For some fixed $\theta \in (0, 1)$ which will be determined later, denote
\[ J = \{  i: \, f_1 \geq \theta f_i, \quad \kappa_i < \nu^{n+1}  \},\quad \quad L = \{  i: \, f_1 < \theta f_i, \quad \kappa_i < \nu^{n+1}  \}. \]
The second line of \eqref{eq2G-6} can be estimated as follows.

\[ \begin{aligned} & \frac{2 \,\kappa_1^2}{(\nu^{n+1} - a)^2}\,\sum\limits_{i \geq 2} \frac{f_i - f_1}{\kappa_1 - \kappa_i}\, \frac{u_i^2}{u^2} (\nu^{n+1} - \kappa_i)^2
+ \frac{2 \kappa_1}{\nu^{n+1} - a}  \sum f_i \frac{u_i^2}{u^2} \big(\kappa_i - \nu^{n+1}\big) \\
\geq \, & \frac{2 \kappa_1^2}{(\nu^{n+1} - a)^2}\sum\limits_{i \in L} \frac{f_i - f_1}{\kappa_1 - \kappa_i} \frac{u_i^2}{u^2} (\nu^{n+1} - \kappa_i)^2
 +  \frac{2 \kappa_1}{\nu^{n+1} - a} \big( \sum\limits_{i \in L} +  \sum\limits_{i \in J} \big) \frac{f_i u_i^2}{u^2} \big(\kappa_i - \nu^{n+1}\big)
 \\
\geq \, & \frac{2 (1 - \theta)\kappa_1}{(\nu^{n+1} - a)^2}\sum\limits_{i \in L} \frac{f_i u_i^2}{u^2} (\nu^{n+1} - \kappa_i)^2
+  \frac{2 \kappa_1}{\nu^{n+1} - a}  \sum\limits_{i \in L} \frac{f_i u_i^2}{u^2} \big(\kappa_i - \nu^{n+1}\big)
- \frac{2}{\theta a}  \sum f_i \kappa_i  \\
= \, &
 \frac{2 \kappa_1}{\nu^{n+1} - a} \sum\limits_{i \in L} \frac{f_i u_i^2}{u^2} \Big(\frac{(\nu^{n+1} - \kappa_i)^2}{\nu^{n+1} - a}  + \kappa_i - \nu^{n+1}\Big)
\\ \,& - \frac{ 2 \,\theta \kappa_1}{(\nu^{n+1} - a)^2}\sum\limits_{i \in L} \frac{f_i u_i^2}{u^2} (\nu^{n+1} - \kappa_i)^2
- \frac{2}{\theta a} \sum f_i \kappa_i \\
\geq \, &
 - \frac{2 \kappa_1}{\nu^{n+1} - a} \sum\limits_{i \in L} \frac{f_i u_i^2}{u^2} \cdot \frac{\nu^{n+1} + a}{\nu^{n+1} - a} \, \kappa_i
- \frac{4 \theta \kappa_1}{ a (\nu^{n + 1} - a)} \sum f_i  \big( 1 + \kappa_i^2 \big) - \frac{2}{\theta a} \sum f_i \kappa_i \\
\geq \, &
- \frac{4 \theta \kappa_1}{a (\nu^{n+1} - a)} \sum f_i  \big( 1 + \kappa_i^2 \big) - \Big( \frac{2}{\theta a} + \frac{4 \kappa_1}{a^2} \Big) \sum f_i \kappa_i.
\end{aligned}
\]
Here we have applied $\tilde{g}^{kl} u_k u_l  = \frac{\delta_{kl}}{u^2} u_k u_l = 1 - (\nu^{n+1})^2$ due to \eqref{eq0-1} in deriving the above inequality.
Choosing $\theta = \frac{a^2}{4}$ and taking the above inequality into \eqref{eq2G-6}, we obtain an upper bound for $\kappa_1$.

\vspace{4mm}

\section{Existence of Strictly Locally Convex Solutions to \eqref{eqn10}}

\vspace{4mm}

The convexity of solutions is a very important prerequisite in this paper, due to the following two reasons: first, the $C^2$ boundary estimates derived in  section 3 require the condition of convexity; second, the $C^2$ interior estimates for prescribed scalar curvature equations in section 6 need certain convexity assumption (see \cite{GQ17}). Therefore, the preservation of convexity of solutions is vital in order to perform the continuity process. In this section, we first give a constant rank theorem in hyperbolic space (see \cite{CLW18, KL87, GM03, GLM06}).

\vspace{2mm}

\begin{thm} \label{Theorem5-1}
Let $\Sigma$ be a $C^4$ oriented connected hypersurface in $\mathbb{H}^{n+1}$ satisfying the prescribed curvature equation
\begin{equation}  \label{eq5-1}
\sigma_k (\kappa) \, = \, \Psi( x_1, \ldots, x_n, u ) > 0.
\end{equation}
Assume that the second fundamental form $\{h_{ij}\}$ on $\Sigma$ is positive semi-definite, and for any ${\bf x} \in \Sigma$ and a local orthonormal frame  $\tau_1, \ldots, \tau_n$ around ${\bf x}$ with $\{ h_{ij} ({\bf x}) \}$ diagonal,
\begin{equation} \label{eq5-20}
\sum\limits_{i \in B} \Big( \Psi_{ii} - \frac{k+1}{k}\,\frac{\Psi_i^2}{\Psi}  +  k \,\Psi \Big) ({\bf x}) \,\lesssim \, 0,
\end{equation}
where the symbol $\lesssim$ is defined in \cite{GM03} and $B$ is the set of bad indices of ${\bf x}$.
Then the second fundamental form on $\Sigma$ is of constant rank.
\end{thm}

\vspace{2mm}

Let $\Sigma$ be a locally convex hypersurface to equation \eqref{eq5-1} for $k < n$ with boundary $\partial \Sigma$. If we can find a condition (we call it {\bf Condition I}) to guarantee that $\Sigma$ is strictly locally convex in a neighbourhood of the boundary $\partial \Sigma$, then together with condition \eqref{eq5-20} in Theorem \ref{Theorem5-1}, we can prove that $\Sigma$ is strictly locally convex up to the boundary.
However, we did not find a suitable Condition I. Still, we proceed to prove the existence as if we have had Condition I in order to show how \eqref{eq5-20} and Condition I play the roles in the continuity process.

\vspace{4mm}

Now we prove the existence.  We use the geometric quantities in section 2 which are expressed in terms of $u$ and write equation \eqref{eq1-1} as
\begin{equation} \label{eqn17}
G( D^2 u, \,D u, u ) = \,F ( a_{ij} ) = \,f( \lambda ( a_{ij} ) ) =  \, \sigma_k^{1/k} (\kappa) = \,\psi^{1/k} ( x,\, u ).
\end{equation}
For convenience, denote
\[ G[u] = \, G (D^2 u, D u, u), \quad  G^{ij}[u] = G^{ij} (D^2 u, D u, u), \quad \mbox{etc.}\]
Let $\delta$ be a small positive constant such that
\begin{equation} \label{eq3-14}
G[\underline{u}] = \, G( D^2 \underline{u}, \,D \underline{u}, \,\underline{u} ) \,>  \delta \,\underline{u} \quad\mbox{in}\quad \Omega_{\epsilon}.
\end{equation}
For $t \in [0, 1]$, consider the following two auxiliary equations.
\begin{equation} \label{eq3-12}
\left\{ \begin{aligned} G (D^2 u, D u, u) \, =  & \, \Big(  ( 1 - t ) \frac{\underline{u}}{ G[\underline{u}]} + t \,\delta^{-1} \Big)^{-1}  \,u \quad\quad & \mbox{in} \quad \Omega_{\epsilon}, \\ u \,  = & \,\, \epsilon \quad \quad & \mbox{on} \quad \Gamma_{\epsilon}. \end{aligned} \right.
\end{equation}

\begin{equation} \label{eq3-13}
\left\{ \begin{aligned} G (D^2 u, D u, u) \, =  & \,\,\Big( ( 1 - t ) \,\delta^{-1} \,u^{-1}  +  t \, \psi^{- 1/k}(x, u) \Big)^{-1} \quad\quad & \mbox{in} \quad \Omega_{\epsilon}, \\ u \,  = & \,\, \epsilon \quad \quad & \mbox{on} \quad \Gamma_{\epsilon}. \end{aligned} \right.
\end{equation}

\begin{lemma} \label{Lemma6-1}
Let $\psi(x)$ be a positive function defined on $\overline{\Omega_{\epsilon}}$. For $x \in \overline{\Omega_{\epsilon}}$ and a positive $C^2$ function $u$ which is strictly locally convex near $x$, if
\[G [u] (x) =  F ( a_{ij}[u] )(x) = f (\kappa)(x) = \psi(x) \,u, \]
then
\[ G_u [u] (x) - \,\psi(x) \,  < 0. \]
\end{lemma}
\begin{proof}
By direct calculation,
\begin{equation*}
G_u  = F^{ij} \frac{1}{w} \gamma^{ik} u_{k l} \gamma^{lj} = \frac{1}{u} \Big( \sum f_i \kappa_i - \frac{1}{w} \sum f_i \Big).
\end{equation*}
Since $ \sum f_i \kappa_i \leq \psi(x) \,u$ by the concavity of $f$ and $f(0) = 0$,
\[ G_u [ u ] (x) - \,\psi(x) \, \leq \,  - \frac{1}{w u}\, \sum f_i  < 0. \]
\end{proof}

\begin{lemma}  \label{Lemma6-2}
For any $t \in [0, 1]$, if $\underline{U}$ and $u$ are respectively any positive strictly locally convex subsolution and solution of \eqref{eq3-12}, then $u \geq \underline{U}$. In particular, the Dirichlet problem \eqref{eq3-12} has at most one strictly locally convex solution.
\end{lemma}
\begin{proof}
We only need to prove that $u \geq \underline{U}$ in $\Omega_{\epsilon}$. If not, then $\underline{U} - u$ achieves a positive maximum at $x_0 \in \Omega_{\epsilon}$, at which,
\begin{equation} \label{eq3-15}
\underline{U}(x_0) > u(x_0),\quad D \underline{U}(x_0) = D u(x_0), \quad D^2\underline{U}(x_0) \leq D^2 u(x_0).
\end{equation}
Note that for any $s \in [0, 1]$, the deformation $u[s] := s \,\underline{U} + (1 - s)\, u$ is strictly locally convex near $x_0$. This is because at $x_0$,
\[ \begin{aligned}
   & \delta_{ij} \,+ \,u[s] \cdot\,{\gamma}^{ik} \big[ u[s] \big] \,\cdot ( u[s] )_{kl}\,\cdot \gamma^{lj}\big[ u[s] \big] \,
  \geq \, \delta_{ij} \,+ \,u[s] \, \,{\gamma}^{ik} [ \underline{U} ] \,\cdot \underline{U}_{kl}\,\cdot \gamma^{lj}[\underline{U}] \\
   = \,& (1 - s) \Big( 1 - \frac{u}{\underline{U}}\Big) \delta_{ij} + \frac{u[s]}{\underline{U}} \Big( \delta_{ij} + \underline{U} \cdot \gamma^{ik}[\underline{U}] \,\cdot \underline{U}_{kl}\,\cdot \gamma^{lj}[\underline{U}] \Big) > 0.
  \end{aligned}
\]
Denote
\begin{equation} \label{eq5-21}
\theta (x, t) = \Big(  ( 1 - t ) \frac{\underline{u}}{ G[\underline{u}]} + t \,\delta^{-1} \Big)^{-1}
\end{equation}
and define a differentiable function of $s \in [0, 1]$:
\[ a(s): = G \Big[ u[s] \Big] (x_0) \,-\, \theta (x_0, t)\,  \,u[s](x_0).  \]
Note that
\[ a(0) = G[u](x_0) \,-\, \, \theta (x_0, t)\,  \,u(x_0) \, = 0 \]
and
\[ a(1) = G [\underline{U}](x_0) \,-\,  \theta (x_0, t)\,  \,\underline{U} (x_0) \,  \geq 0.\]
Thus there exists $s_0 \in [0, 1]$ such that $a(s_0) = 0$ and $a'(s_0) \geq 0$, i.e.,
\begin{equation} \label{eq3-16}
 G\big[ u[s_0] \big] (x_0) \,=\, \theta (x_0, t)  \, u[s_0] (x_0)
\end{equation}
and
\begin{equation} \label{eq3-17}
\begin{aligned}
& G^{ij}\big[ u[s_0]  \big](x_0) \,\, D_{ij}  (\underline{U} - u)(x_0)
 + G^i \big[ u[s_0]  \big](x_0)\,\,  D_i  (\underline{U} - u)(x_0)
 \\ & +  \Big(G_u \big[ u[s_0] \big](x_0) - \theta (x_0, t)\, \Big)  (\underline{U} - u)(x_0) \geq 0.
\end{aligned}
\end{equation}
However, the above inequality can not hold by \eqref{eq3-15}, \eqref{eq3-16} and Lemma \ref{Lemma6-1}.
\end{proof}

\vspace{2mm}

\begin{thm} \label{Theorem6-1}
Under assumption \eqref{eqn12} and Condition I, for any $t \in [0, 1]$, the Dirichlet problem \eqref{eq3-12} has a unique strictly locally convex solution $u$, which satisfies $u \geq \underline{u}$ in $\Omega_{\epsilon}$.
\end{thm}
\begin{proof}
Uniqueness is proved in Lemma \ref{Lemma6-2}. For existence of a strictly locally convex solution, we first verify that $\Psi = (\theta (x, t) \, u)^k = \Theta (x, t) \,u^k$ satisfies condition \eqref{eq5-20} in the constant rank theorem.
By direct calculation,
\[
\begin{aligned}
 & \Psi_{ii} - \frac{k+1}{k}\,\frac{\Psi_i^2}{\Psi}  +  k \,\Psi  \\
 = \,&  \sum\limits_{\alpha, \beta = 1}^n \,\Big( \Theta_{x_{\alpha} x_{\beta}} - \frac{k+1}{k} \,\frac{\Theta_{x_\alpha} \Theta_{x_\beta}}{\Theta} \Big) (x_{\alpha})_i (x_{\beta})_i \,u^k  + \sum\limits_{\alpha = 1}^n \Theta_{x_{\alpha}} (x_{\alpha})_{ii} \,u^k \\ & - 2 \sum\limits_{\alpha = 1}^n \Theta_{x_{\alpha}} (x_{\alpha})_i \, u^{k-1} u_i  - 2 k\, \Theta u^{k-2} u_i^2  + \Theta \, k\, u^{k-1} u_{ii} + k\, \Theta \,u^k.
\end{aligned}
\]
By \eqref{eqC-3}, \eqref{eq0-2}, \eqref{eq0-3} and \eqref{eq0-1}, for $i \in B$ and $\alpha = 1, \ldots, n$, we have
\begin{equation} \label{eq5-23}
\begin{aligned}
(x_{\alpha})_{ii} \sim & - \nu^{n+1} \,u \,\nu^{\alpha} + \frac{2}{u} (x_\alpha)_i \,u_i - \frac{1}{u} \sum\limits_{l = 1}^n u_l \,(x_{\alpha})_l \\
=  & - u \,(\nu \cdot \partial_{n+1})   (\nu \cdot \partial_{\alpha}) - u \sum\limits_{l = 1}^n \big(\frac{\tau_l}{u} \cdot \partial_{n+1}\big) \,\big(\frac{\tau_l}{u} \cdot \partial_{\alpha}\big) + \frac{2}{u} (x_\alpha)_i \,u_i
\\ = \,& \frac{2}{u} (x_\alpha)_i \,u_i
\end{aligned}
\end{equation}
and
\begin{equation} \label{eq5-22}
u_{ii} \,\sim \,  \frac{2}{u}  \,u_i^2 - u.
\end{equation}
Therefore by \eqref{eqn12},
\[
\sum\limits_{i \in B} \Big(\Psi_{ii} - \frac{k+1}{k}\,\frac{\Psi_i^2}{\Psi}  +  k \,\Psi \Big) \sim \, - k\, \Theta^{ \frac{1}{k} + 1}
\sum\limits_{i \in B}  \sum\limits_{\alpha, \beta = 1}^n \,\Big( \Theta^{- \frac{1}{k}} \Big)_{x_{\alpha} x_{\beta}} (x_{\alpha})_i (x_{\beta})_i \,u^k  \leq 0.
\]

Next, we use the standard continuity method to prove the existence. Note that $\underline{u}$ is a subsolution of \eqref{eq3-12} by \eqref{eq3-14}. We have obtained the $C^2$ bound for strictly locally convex solution $u$ (satisfying $u \geq \underline{u}$ by Lemma \ref{Lemma6-2}) of \eqref{eq3-12}, which implies the uniform ellipticity of equation \eqref{eq3-12}.
By Evans-Krylov theory \cite{Evans, Krylov}, we obtain the $C^{2, \alpha}$ estimate which is independent of $t$,
\begin{equation} \label{eq3-20}
\Vert u \Vert_{C^{2, \alpha} (\, \overline{ \Omega_{\epsilon} } \,)}  \leq C.
\end{equation}
Denote
\[ C_0^{2, \alpha} (\, \overline{ \Omega_{\epsilon} } \,) := \{ w \in C^{2, \alpha}( \, \overline{ \Omega_{\epsilon} } \, ) \,| \,w = 0 \,\, \mbox{on} \,\, \Gamma_{\epsilon} \}, \]

\[ \mathcal{U} := \left\{ w \in C_0^{2, \alpha} (\, \overline{ \Omega_{\epsilon} } \,) \,\Big| \, \underline{u} + w \,\,\mbox{is}\,\,\mbox{strictly}\,\,\mbox{locally}\,\,\mbox{convex}\,\,\mbox{in} \,\,\overline{\Omega_{\epsilon}}  \right\}. \]
We can see that $C_0^{2, \alpha} (\, \overline{ \Omega_{\epsilon} } \,)$ is a subspace of $C^{2, \alpha}( \,\overline{ \Omega_{\epsilon} }\, )$ and
$\mathcal{U}$ is an open subset of $C_0^{2, \alpha} (\,\overline{ \Omega_{\epsilon} }\,)$.
Consider the map $\mathcal{L}: \,\mathcal{U} \times [ 0, 1 ] \rightarrow C^{\alpha}(\, \overline{ \Omega_{\epsilon} } \,)$,
\[ \mathcal{L} ( w, t ) = G [ \underline{u} + w ] \, - \, \theta (x, t)  \,(\underline{u} + w).  \]
Set
\[ \mathcal{S} = \{ t \in [0, 1] \,|\, \mathcal{L}(w, t) = 0 \,\,\mbox{has}\,\,\mbox{a}\,\,\mbox{solution}\,\,w \,\,\mbox{in}\,\,\mathcal{U}\, \}. \]
Note that $\mathcal{S} \neq \emptyset$ since $\mathcal{L}(0, 0) = 0$.

We claim that $\mathcal{S}$ is open in $[0, 1]$. In fact, for any $t_0 \in \mathcal{S}$, there exists $w_0 \in \mathcal{U}$ such that $\mathcal{L} ( w_0, t_0 ) = 0$. The Fr\'echet derivative of $\mathcal{L}$ with respect to $w$ at $(w_0, t_0)$ is a linear elliptic operator from $C^{2, \alpha}_0 (\, \overline{\Omega_{\epsilon}} \,)$ to $C^{\alpha}(\, \overline{\Omega_{\epsilon}} \,)$,
\[
\mathcal{L}_w \big|_{(w_0, t_0)} ( h )  =  \,  G^{ij}[\underline{u} + w_0]\, D_{ij} h  +  G^i [ \underline{u} + w_0 ]\, D_i h  \\ + \Big(G_u [ \underline{u} + w_0] - \theta (x, t_0)  \Big) h.
\]
By Lemma \ref{Lemma6-1}, $\mathcal{L}_w \big|_{(w_0, t_0)}$ is invertible. By implicit function theorem, a neighborhood of $t_0$ is also contained in $\mathcal{S}$.

Next, we show that $\mathcal{S}$ is closed in $[0, 1]$. Let $t_i$ be a sequence in $\mathcal{S}$ converging to $t_0 \in [0, 1]$ and $w_i \in \mathcal{U}$ be the unique (by Lemma \ref{Lemma6-2}) solution corresponding to $t_i$, i.e. $\mathcal{L} (w_i, t_i) = 0$.  By Lemma \ref{Lemma6-2}, $w_i \geq 0$. By \eqref{eq3-20}, $u_i := \underline{u} + w_i$ is a bounded sequence in $C^{2, \alpha}(\,\overline{\Omega_{\epsilon}}\,)$, which possesses a subsequence converging to a locally convex solution $u_0$ of \eqref{eq3-12}. By Condition I and Theorem \ref{Theorem5-1}, we know that $u_0$ is strictly locally convex in $\overline{\Omega_{\epsilon}}$. Since $w_0 := u_0 - \underline{u} \in \mathcal{U}$ and $\mathcal{L}(w_0, t_0) = 0$, thus $t_0 \in \mathcal{S}$.
\end{proof}

From now on we may assume $\underline{u}$ is not a solution of \eqref{eqn10}, since otherwise we are done.

\begin{lemma} \label{Lemma6-3}
If $u \geq \underline{u}$ is a strictly locally convex solution of \eqref{eq3-13} in $\Omega_{\epsilon}$, then
$u > \underline{u}$ in $\Omega_{\epsilon}$ and $(u - \underline{u})_{\gamma} > 0$ on $\Gamma_{\epsilon}$.
\end{lemma}

\begin{proof}
To keep the strict local convexity of the variations in our proof, we rewrite \eqref{eq3-13} in terms of $v$,
\begin{equation}  \label{eq3-18}
\left\{ \begin{aligned}
G(D^2 v, D v, v) = & \,\,\psi^t(x, v)  \quad \quad  &\mbox{in} \quad & \Omega_{\epsilon}, \\
v = & \,\,\epsilon^2  \quad \quad & \mbox{on} \quad & \Gamma_{\epsilon}. \end{aligned} \right.
\end{equation}
Since $\underline{u}$ is a subsolution but not a solution of \eqref{eq3-13}, equivalently, $\underline{v}$ is a subsolution but not a solution of \eqref{eq3-18}, thus,
\begin{equation} \label{eq3-19}
 G[\underline{v}] - G[v] \geq \psi^t(x, \underline{v}) - \psi^t(x, v).
\end{equation}
Denote $v[s] := s \,\underline{v} + (1 - s)\, v$, which is strictly locally convex over $\Omega_{\epsilon}$ for any $s \in [0, 1]$ since
\[  \delta_{ij} + \frac{1}{2} \big( v[s] \big)_{ij} = \,s \Big(  \delta_{ij} + \frac{1}{2}\, \underline{v}_{ij} \Big) + ( 1 - s ) \Big(  \delta_{ij} + \frac{1}{2}\, v_{ij} \Big) > 0 \quad \mbox{in} \quad \Omega_{\epsilon}. \]
From \eqref{eq3-19} we can deduce that
\[ a_{ij}(x) D_{ij}( \underline{v} - v ) + b_i(x) D_i(\underline{v} - v) + c(x) (\underline{v} - v) \geq 0 \quad \quad  \mbox{in} \quad  \Omega_{\epsilon}, \]
where
\[ a_{ij}(x) = \int_0^1 G^{ij}\big[ v[s] \big] (x) \,d s, \quad b_{i}(x) = \int_0^1 G^{i}\big[ v[s] \big] (x) \,d s,\]
\[ c(x) = \int_0^1 G_v \big[ v[s] \big] (x) - {\psi^t}_v (x, v[s] ) \,d s. \]
Applying the Maximum Principle and Lemma H (see p. 212 of \cite{GNN}) we conclude that $v > \underline{v}$ in $\Omega_{\epsilon}$ and $(v - \underline{v})_{\gamma} > 0$ on $\Gamma_{\epsilon}$. Hence the lemma is proved.
\end{proof}

\begin{thm} \label{Theorem6-2}
Under assumption \eqref{eqn12}, \eqref{eqn13} and Condition I, for any $t \in [0, 1]$, the Dirichlet problem \eqref{eq3-13} possesses a strictly locally convex solution satisfying $u \geq \underline{u}$ in $\Omega_{\epsilon}$. In particular, the Dirichlet problem \eqref{eqn10}  has a strictly locally convex solution $u^{\epsilon}$ satisfying $u^{\epsilon} \geq \underline{u}$ in $\Omega_{\epsilon}$.
\end{thm}

\begin{proof}
We first verify that
\[\Psi = \,\Big( ( 1 - t ) \,\delta^{-1} \,u^{-1}  +  t \, \psi^{- 1/k}(x, u) \Big)^{-k}\]
satisfies condition \eqref{eq5-20} in the constant rank theorem. In fact, by assumption \eqref{eqn13},  \eqref{eq5-23} and \eqref{eq5-22},
\[\begin{aligned}
& k \,\psi^{ \frac{1}{k} + 1} \sum\limits_{i \in B} \,\left( \big(\psi^{- \frac{1}{k}}\big)_{ii} - \psi^{- \frac{1}{k}} \right) \\
\sim \, & \sum\limits_{i \in B} \tau_i^{T}
 \left(
         \begin{array}{cc}
           \frac{k+1}{k} \frac{\psi_{x_\alpha} \psi_{x_\beta}}{\psi} - \psi_{x_{\alpha} x_{\beta}}  + \frac{u \psi_u - k \psi}{u^2} \delta_{\alpha\beta}
            & \frac{k+1}{k} \frac{\psi_{x_\alpha} \psi_{u}}{\psi} - \psi_{x_{\alpha} u} - \frac{\psi_{x_\alpha}}{u} \\
           \frac{k+1}{k} \frac{\psi_{x_\alpha} \psi_{u}}{\psi} - \psi_{x_{\alpha} u} - \frac{\psi_{x_\alpha}}{u} &
          \frac{k+1}{k} \frac{\psi_{u}^2}{\psi} - \psi_{u u} - \frac{k \,\psi}{u^2} - \frac{\psi_u}{u} \\
         \end{array}
       \right) \tau_i  \geq  0,
\end{aligned} \]
and consequently,
\[ \begin{aligned}
& \sum\limits_{i \in B} \Big(\Psi_{ii} - \frac{k+1}{k}\,\frac{\Psi_i^2}{\Psi}  +  k \,\Psi \Big)
\\ = \,&
- k\, \Psi^{\frac{ k + 1}{k}} \sum\limits_{i \in B} \left( (1 - t) \delta^{-1} \Big( (u^{-1})_{ii} - u^{-1} \Big) + t \Big( (\psi^{-1/k})_{ii} - \psi^{-1/k} \Big) \right)\,  \lesssim 0.
\end{aligned}
\]

We have established $C^{2, \alpha}$ estimates for strictly locally convex solutions $u \geq \underline{u}$ of \eqref{eq3-13}, which further imply $C^{4, \alpha}$ estimates by classical Schauder theory,
\begin{equation} \label{eq3-21}
\Vert u \Vert_{C^{4,\alpha}(\overline{\Omega_{\epsilon}})} < C_4.
\end{equation}
In addition, we have
\begin{equation} \label{eq3-23}
\mbox{dist}(\kappa[u], \partial\Gamma_k) > c_2 > 0 \quad \mbox{in} \,\, \overline{\Omega_{\epsilon}},
\end{equation}
where $C_4$, $c_2$ are independent of $t$. Denote
\[ C_0^{ 4, \alpha} (\,\overline{\Omega_{\epsilon}}\,) := \{ w \in C^{ 4, \alpha}( \,\overline{\Omega_{\epsilon}}\, ) \,| \,w = 0 \,\, \mbox{on} \,\, \Gamma_{\epsilon} \} \]
and
\[ \mathcal{O} := \left\{ w \in C_0^{4, \alpha} (\overline{\Omega_{\epsilon}}) \,\left\vert\,\begin{footnotesize}\begin{aligned} & w > 0 \,\,\mbox{in}\,\,\Omega_{\epsilon}, \quad  w_{\gamma} > 0 \,\,\mbox{on}\,\, \Gamma_{\epsilon},  \quad \Vert w {\Vert}_{C^{4,\alpha}(\overline{\Omega_{\epsilon}})} < C_4 + \Vert\underline{u}\Vert_{C^{4,\alpha}(\overline{\Omega_{\epsilon}})}\\ &  \{ \delta_{ij} + (\underline{u} + w)_i (\underline{u} + w)_j + (\underline{u} + w) (\underline{u} + w)_{ij} \} > 0   \,\, \mbox{in} \,\, \overline{\Omega_{\epsilon}}, \\
&  \mbox{dist}(\kappa[\underline{u} + w], \partial\Gamma_k) > c_2  \,\, \mbox{in} \,\, \overline{\Omega_{\epsilon}} \end{aligned}\end{footnotesize} \right.\right\}, \]
which is a bounded open subset of $C_0^{ 4, \alpha} (\,\overline{\Omega_{\epsilon}}\,)$.
Define
$\mathcal{M}_t (w):  \,\mathcal{O} \times [ 0, 1 ] \rightarrow C^{2,\alpha}(\overline{\Omega_{\epsilon}})$,
\[ \mathcal{M}_t (w) = G [ \underline{u} + w ] \, -\, \Big(( 1 - t ) \,\delta^{ - 1} \cdot (\underline{u} + w)^{-1}  + \, t \, \psi^{-1/k}(x, \underline{u} + w) \Big)^{-1}. \]
Let $u^0$ be the unique strictly locally convex solution of \eqref{eq3-12} at $t = 1$ (the existence and uniqueness are guaranteed by Theorem \ref{Theorem6-1} and Lemma \ref{Lemma6-2}). Observe that $u^0$ is also the unique solution of \eqref{eq3-13} when $t = 0$. By Lemma \ref{Lemma6-2},  $w^0:  = u^0 - \underline{u} \geq 0$ in $\Omega_{\epsilon}$. By Lemma \ref{Lemma6-3}, $w^0 > 0$ in $\Omega_{\epsilon}$ and ${w^0}_{\gamma} > 0$ on $\Gamma_{\epsilon}$. Also, $\underline{u} + w^0$ satisfies \eqref{eq3-21} and \eqref{eq3-23}. Thus, $w^0 \in \mathcal{O}$. By Condition I, Theorem \ref{Theorem5-1},
Lemma \ref{Lemma6-3}, \eqref{eq3-21} and \eqref{eq3-23}, $\mathcal{M}_t(w) = 0$ has no solution on $\partial\mathcal{O}$ for any $t \in [0, 1]$. Besides, $\mathcal{M}_t$ is uniformly elliptic on $\mathcal{O}$ independent of $t$. Therefore, we can define the $t$-independent degree of $\mathcal{M}_t$ on $\mathcal{O}$ at $0$:
\[ \deg (\mathcal{M}_t, \mathcal{O}, 0).  \]
To find this degree, we only need to compute
$ \deg (\mathcal{M}_0, \mathcal{O}, 0) $.
By the above discussion, we know that $\mathcal{M}_0 ( w ) = 0$ has a unique solution $w^0 \in \mathcal{O}$. The Fr\'echet derivative of $\mathcal{M}_0$ with respect to $w$ at $w^0$ is a linear elliptic operator from $C^{4, \alpha}_0 (\overline{\Omega_{\epsilon}})$ to $C^{2, \alpha}(\overline{\Omega_{\epsilon}})$,
\begin{equation}
\mathcal{M}_{0,w} |_{w^0} ( h )  =  \,  G^{ij}[ u^0 ]\, D_{ij} h  + G^i [ u^0 ] \,D_i h  \, + ( G_u [ u^0 ]  - \,\delta \, ) h.
\end{equation}
By Lemma \ref{Lemma6-1}, $G_u [ u^0 ]  - \,\delta \, < 0$ in $\overline{\Omega_{\epsilon}}$ and thus $\mathcal{M}_{0,w} |_{w^0}$ is invertible. By the degree theory established in \cite{Li89},
\[ \deg (\mathcal{M}_0, \mathcal{O}, 0) = \deg( \mathcal{M}_{0, w^0}, B_1, 0 ) = \pm 1 \neq 0, \]
where $B_1$ is the unit ball in $C_0^{4,\alpha}(\overline{\Omega_{\epsilon}})$. Thus
$\deg(\mathcal{M}_t, \mathcal{O}, 0) \neq 0$ for all $t \in [0, 1]$, which implies that the Dirichlet problem \eqref{eq3-13} has at least one strictly locally convex solution $u \geq \underline{u}$ for any $t \in [0, 1]$.
\end{proof}

\vspace{4mm}

\section{\large Interior second order estimates for prescribed scalar curvature equations in $\mathbb{H}^{n+1}$}

\vspace{4mm}

Let $u^{\epsilon} \geq \underline{u}$ be a strictly locally convex solution over $\Omega_{\epsilon}$ to the Dirichlet problem \eqref{eqn10}. For any fixed $ \epsilon_0 > 0$, we want to establish the uniform $C^2$ estimates for $u^{\epsilon}$ for any $0 < \epsilon < \frac{\epsilon_0}{4}$ on $\overline{\Omega_{\epsilon_0}}$,  namely,
\begin{equation} \label{eq4-1}
\Vert u^{\epsilon} \Vert_{C^2(\,\overline{\Omega_{\epsilon_0}}\,)} \leq C, \quad\quad\quad\forall \quad 0 < \epsilon < \frac{\epsilon_0}{4}.
\end{equation}
In what follows, let $C$ be a positive constant which is independent of $\epsilon$ but depends on $\epsilon_0$.
By \eqref{eqn15}, we immediately obtain the uniform $C^0$ estimate:
\begin{equation} \label{eq4-2}
 \epsilon_0 \,\leq \,  u^{\epsilon} \, \leq \, C \quad \mbox{on} \quad \overline{\Omega_{\epsilon_0}}, \quad\quad\quad\forall \quad 0 < \epsilon < \epsilon_0.
\end{equation}

For uniform $C^1$ estimate on $\overline{\Omega_{\epsilon_0}}$, we make use of the Euclidean strict local convexity of $(u^\epsilon)^2 + |x|^2$ (see \cite{Tru84} for a similar idea) to obtain
\[
\max\limits_{\overline{\Omega_{\epsilon_0}}} \big\vert D \big((u^\epsilon)^2 + |x|^2\big) \big\vert \leq \frac{C (n) \max\limits_{\overline{\Omega_{\epsilon_0 / 2}}} \big( (u^\epsilon)^2 + |x|^2 \big)}{\mbox{dist} (\Gamma_{\epsilon_0 / 2}, \overline{\Omega_{\epsilon_0}})}, \quad\quad\quad\forall \quad 0 < \epsilon < \frac{\epsilon_0}{2}.
\]
It follows that,
\begin{equation} \label{eq4-3}
 \Vert u^{\epsilon} \Vert_{C^1 (\,\overline{\Omega_{\epsilon_0}})} \,\leq \, C, \quad\quad\quad\forall \quad 0 < \epsilon < \frac{\epsilon_0}{2}.
\end{equation}

We are now in a position to prove
\begin{equation} \label{eq4-10}
 \big| D^2 u^{\epsilon} \big|\,\leq \, C \quad \mbox{on} \quad \overline{\Omega_{\epsilon_0}}, \quad\quad\quad\forall \quad 0 < \epsilon < \frac{\epsilon_0}{4},
\end{equation}
which is equivalent to
\begin{equation} \label{eq4-11}
\max\limits_{\overline{\Omega_{\epsilon_0}}} \big| \kappa_i [ u^{\epsilon}] \big|\,\leq \, C, \quad\quad\quad\forall \quad 0 < \epsilon < \frac{\epsilon_0}{4}.
\end{equation}

Choose $r = \mbox{dist}(\overline{\Omega_{\epsilon_0}}, \Gamma_{\epsilon_0 / 2})$, and cover $\overline{\Omega_{\epsilon_0}}$ by finitely many open balls $B_{\frac{r}{2}}$ with radius $\frac{r}{2}$ and centered in $\Omega_{\epsilon_0}$. Note that the number of such open balls depends on $\epsilon_0$. In addition, the corresponding balls $B_r$ are all contained in $\Omega_{\epsilon_0 / 2}$, over which, we are able to apply the gradient estimate due to \eqref{eq4-3}:
\[\Vert u^{\epsilon} \Vert_{C^1 (\,\overline{\Omega_{\epsilon_0 / 2}})} \,\leq \, C, \quad\quad\quad\forall \quad 0 < \epsilon < \frac{\epsilon_0}{4}.
\]
If we are able to establish the following interior $C^2$ estimate on each $B_r$:
\[ \sup\limits_{B_{r/2}} \, \big| \kappa_i [u^{\epsilon}] \big| \,\leq \, C (\Vert u^{\epsilon} \Vert_{C^1 (B_r)} ), \quad\quad\quad\forall \quad 0 < \epsilon < \frac{\epsilon_0}{4}, \]
then \eqref{eq4-11} can be proved. Since the principal curvatures $\kappa_i [u^{\epsilon}]$, $i = 1, \ldots, n$ and the gradient $D u^{\epsilon}$ are invariant under the change of Euclidean coordinate system, we may assume the center of $B_r$ is $0$. For convenience, we also omit the superscript in $u^{\epsilon}$ and write as $u$.

In what follows, we will use Guan-Qiu's idea \cite{GQ17} to derive the interior $C^2$ estimate
\begin{equation} \label{eqn18}
\sup\limits_{B_{r/2}} \, | \kappa_i (x) | \,\leq \, C
\end{equation}
for strictly locally convex hypersurface $\Sigma$ in $\mathbb{H}^{n+1}$ to the following equation
\begin{equation} \label{eq4-4}
\sigma_2 ( \kappa ) = \,\psi ( {\bf x} ),
\end{equation}
where $B_r \subset \mathbb{R}^n$ is the open ball with radius $r$ centered at $0$ and $C$ is a positive constant depending only on $n$, $r$, $\Vert \Sigma \Vert_{C^1( B_r )}$, $\Vert \psi \Vert_{C^2(B_r)}$ and $\inf_{B_r} \psi$.

For $x \in B_r$ and $\xi \in \mathbb{S}^{n-1} \cap T_{(x, u)} \Sigma$, consider the test function
\begin{equation*}
\Theta (x, u, \xi) =  \, 2 \ln \rho(x) + \alpha \Big(\frac{u}{\nu^{n+1}} \Big)^2 - \beta \left(\frac{{\bf x} \cdot \nu}{{\nu}^{n+1}}\right) +  \ln \ln  h_{\xi \xi},
\end{equation*}
where $\rho(x) = r^2 - |x|^2$ with $|x|^2 = \sum_{i = 1}^n x_i^2$ and $\alpha$, $\beta$ are positive constants to be determined later. At this point, we remind the readers that $\cdot$ means the inner product in $\mathbb{R}^{n+1}$ while $\langle\,\, , \,\, \rangle$ represents the inner product in $\mathbb{H}^{n+1}$.

The maximum value of $\Theta$ can be attained in an interior point $x^0 = (x_1, \ldots, x_n) \in B_r$. Let $\tau_1, \ldots, \tau_n$ be a normal coordinate frame around $(x^0, u(x^0))$ on $\Sigma$ and assume the direction obtaining the maximum to be $\xi = {\tau}_1$. By rotation of $\tau_2, \ldots, \tau_n$ we may assume that $\big( h_{ij}(x^0) \big)$ is diagonal. Thus, the function
\[   2 \ln \rho(x) + \alpha \Big(\frac{u}{\nu^{n+1}} \Big)^2 - \beta \,\Big( \frac{{\bf x} \cdot \nu}{{\nu}^{n+1}}  \Big)  +  \ln \ln  h_{11}    \]
also achieves its maximum at $x^0$. Therefore, at $x^0$,

\begin{equation} \label{eq4-6}
\frac{2 \,{\rho}_i}{\rho} + 2 \alpha \frac{u}{\nu^{n + 1}} \Big( \frac{u}{\nu^{n+1}} \Big)_i - \beta  \left(\frac{{\bf x} \cdot \nu}{  {\nu}^{n+1}}\right)_i + \frac{h_{11i}}{h_{11}\, \ln h_{11}} = 0,
\end{equation}

\vspace{2mm}

\begin{equation} \label{eq4-7}
\begin{aligned}
&  \frac{2 \sigma_2^{ii} \rho_{ii}}{\rho} - \frac{2 \sigma_2^{ii} \rho_i^2}{\rho^2}  +  2 \alpha \sigma_2^{ii} \left( \Big( \frac{u}{\nu^{n+1}}  \Big)_i^2 + \Big( \frac{u}{\nu^{n+1}} \Big) \Big( \frac{u}{\nu^{n+1}} \Big)_{ii} \right)
\\ & - \beta \sigma_2^{ii} \left(\frac{{\bf x} \cdot \nu}{{\nu}^{n+1}}\right)_{ii} + \frac{\sigma_2^{ii} h_{11ii}}{h_{11} \ln h_{11}} - (1 + \ln h_{11}) \frac{\sigma_2^{ii} h_{11i}^2}{(h_{11} \ln h_{11})^2} \,\,\leq 0.
\end{aligned}
\end{equation}

To compute the quantities in \eqref{eq4-6} and \eqref{eq4-7}, we first convert them into quantities in $\mathbb{H}^{n+1}$, and apply
the Gauss formula and Weingarten formula
\[ {\bf D}_{\tau_i} \tau_j = \nabla_{\tau_i} \tau_j + h_{ij}\, {\bf n}, \]
\[ {\bf n}_i = - h_{ij} \,\tau_j. \]
We also note that in $\mathbb{H}^{n+1}$,
\[ {\bf D}_{\bf y} \,\partial_{n+1} = \, - \frac{1}{u} \,{\bf y}, \]
where ${\bf y}$ is any vector field in $\mathbb{H}^{n+1}$. This implies that $\partial_{n+1}$ is a conformal Killing field in $\mathbb{H}^{n+1}$.
By straightforward calculation, we obtain
\begin{equation} \label{eq5-12}
\Big( \frac{u}{\nu^{n+1}} \Big)_i \,  = \, \Big( \frac{1}{\langle {\bf n}, \,\partial_{n+1} \rangle} \Big)_i  =  \,\kappa_i \,\frac{\tau_i \cdot \partial_{n+1}}{(\nu^{n+1})^2},
\end{equation}

\begin{equation} \label{eq5-13}
\Big( \frac{u}{\nu^{n+1}} \Big)_{ii} = h_{iij} \frac{\tau_j \cdot \partial_{n+1}}{ (\nu^{n+1})^2} + \kappa_i^2 \frac{u}{\nu^{n+1}} - \frac{u}{(\nu^{n+1})^2} \,\kappa_i  + 2 \kappa_i^2 \frac{(\tau_i \cdot \partial_{n+1})^2}{u (\nu^{n+1})^3 }.
\end{equation}

Now we choose the conformal Killing field ${\bf x}$ in $\mathbb{H}^{n+1}$ to be
\[ {\bf x} = x_{n+1} \,\sum_{i = 1}^n \, x_i \partial_i + \frac{1}{2} \,\Big( x_{n+1}^2 - |x|^2  \Big) \,\partial_{n+1}.  \]
We can verify that
\[ {\bf D}_{\bf y} \, {\bf x} \,=\,\phi \,\,{\bf y}, \quad \quad \quad \phi =  \frac{x_{n+1}^2 + |x|^2}{2\, x_{n+1}}, \]
where ${\bf y}$ is any vector field in $\mathbb{H}^{n+1}$.

Again, by straightforward calculation, we find that
\begin{equation} \label{eq5-8}
\left(\frac{{\bf x} \cdot \nu}{{\nu}^{n+1}}\right)_i = \frac{\kappa_i}{u \, \nu^{n + 1}} \,\left( \frac{ ({\bf x} \cdot \nu) \, (\tau_i \cdot \partial_{n+1} )}{\nu^{n+1}}  - {\bf x} \cdot \tau_i  \right),
\end{equation}

\begin{equation} \label{eq5-9}
\begin{aligned}
\left( \frac{{\bf x \cdot \nu}}{\nu^{ n + 1}} \right)_{ii} = \,&\, - \Big( \frac{ \phi\, u}{\nu^{n + 1}} + \frac{{\bf x} \cdot \nu}{(\nu^{n+1})^2} \Big) \kappa_i + \frac{2 \kappa_i (\tau_i \cdot \partial_{n + 1})}{u \nu^{n+1}} \Big( \frac{{\bf x} \cdot \nu}{\nu^{n+1}} \Big)_i \\
& + \frac{1}{u (\nu^{n+1})^2} \Big( ({\bf x} \cdot \nu) (\tau_j \cdot \partial_{n+1}) - ({\bf x} \cdot \tau_j) \nu^{n + 1}\Big) h_{iij}.
\end{aligned}
\end{equation}

Also, since
\[ |x|^2 \,=\, \frac{1 - 2 \langle {\bf x}, \partial_{n+1} \rangle}{\langle \partial_{n+1}, \partial_{n+1} \rangle}, \]
by direct calculation we obtain
\begin{equation} \label{eq5-10}
\begin{aligned}
 \rho_i =  \,& 2 u^3 \langle \tau_i, \partial_{n+1} \rangle \langle {\bf x}, \partial_{n+1} \rangle -  2 u \langle {\bf x}, \tau_i \rangle
 \\ = \, &  \frac{2}{u} \Big( ( \tau_i \cdot \partial_{n+1}) ( {\bf x} \cdot \partial_{n+1} ) -   {\bf x} \cdot \tau_i  \Big),
 \end{aligned}
\end{equation}

\begin{equation} \label{eq5-11}
\begin{aligned}
\rho_{ii} = \,& \kappa_i \Big( (u^2 - |x|^2) \nu^{n+1} - 2 {\bf x} \cdot \nu \Big)
\\ \, & + \frac{4 u^2 - 2 |x|^2 }{u^2} (\tau_i \cdot \partial_{n+1})^2 - \frac{4}{u^2} (\tau_i \cdot {\bf x})(\tau_i \cdot \partial_{n+1}) - 2 u^2.
\end{aligned}
\end{equation}

Differentiate \eqref{eq4-4} twice,
\begin{equation} \label{eq4-8}
\sigma_2^{ii} h_{iik} \, = \,\psi_k,
\end{equation}
\begin{equation}  \label{eq4-9}
\sum\limits_{i \neq j} h_{ii1} h_{jj1} - \sum\limits_{ i \neq j} h_{i j 1}^2 + \sigma_2^{ii} h_{ii11} = \,\, \psi_{11} \,\geq \, - C \kappa_1.
\end{equation}

Now taking \eqref{eq5-11}, \eqref{eq5-12}, \eqref{eq5-13}, \eqref{eq5-9}, \eqref{eq4-6}, \eqref{eq4-8}, \eqref{eq2G-4}, \eqref{eq4-9} into \eqref{eq4-7}, we obtain
\begin{equation} \label{eq5-14}
\begin{aligned}
- \frac{C}{\rho} \,\sigma_1 - C \alpha - C \beta - \frac{2 \sigma_2^{ii} \rho_i^2}{\rho^2}  + 2 \alpha \frac{u^2}{(\nu^{n+1})^2} \sigma_2^{ii} \kappa_i^2 - \frac{2 \sigma_2^{ii} \kappa_i (\tau_i \cdot \partial_{n+1}) \,h_{11i}}{u \,\nu^{n+1} \kappa_1 \ln \kappa_1}
\\ + \frac{\sum_{i \neq j} h_{ij1}^2 - \sum_{i \neq j} h_{ii1} h_{jj1}}{\kappa_1 \,\ln \kappa_1} - \frac{C \sigma_1}{\ln \kappa_1} - \frac{\sigma_2^{ii} \kappa_i^2}{\ln \kappa_1} - \big( 1 + \ln \kappa_1 \big) \frac{\sigma_2^{ii} h_{11i}^2}{(\kappa_1 \ln \kappa_1)^2} \, \leq 0.
\end{aligned}
\end{equation}
By Theorem 1.2 of \cite{Chen13} (see also Lemma 2 of \cite{GQ17}), we have
\[ - \sum_{i \neq j} h_{ii1} h_{jj1} \geq \, \frac{1}{2 \sigma_2} \,\frac{(n-1) \big(2 \sigma_2 \,h_{111} - \kappa_1 \,\psi_1 \big)^2}{(n - 1) \kappa_1^2 + 2 (n - 2) \sigma_2} - \frac{\psi_1^2}{2 \sigma_2}.\]
Also,
\[ - \frac{2 \sigma_2^{ii} \kappa_i (\tau_i \cdot \partial_{n+1}) \,h_{11i}}{u \,\nu^{n+1} \kappa_1 \ln \kappa_1} \geq \, - \frac{u^2}{(\nu^{n+1})^2} \sigma_2^{ii} \kappa_i^2  - \frac{(\tau_i \cdot \partial_{n+1})^2}{u^4} \, \frac{\sigma_2^{ii} h_{11i}^2}{(\kappa_1 \ln \kappa_1)^2}. \]
Thus, when $\kappa_1$ is sufficiently large,
\eqref{eq5-14} reduces to
\begin{equation} \label{eq5-15}
- \frac{C}{\rho} \sigma_1 - \frac{2 \,\sigma_2^{ii} \,\rho_i^2}{\rho^2} + (2 \alpha - 2) \frac{u^2}{(\nu^{n+1})^2} \sigma_2^{ii} \kappa_i^2 + \frac{\sigma_2^{ii}\, h_{11i}^2}{20 \,\kappa_1^2 \,\ln \kappa_1} \, \leq 0.
\end{equation}

As in \cite{GQ17}, we divide our discussion into three cases. We show all the details to indicate the tiny differences due to the outer space $\mathbb{H}^{n+1}$.

\vspace{4mm}

Case (i):  when $|x|^2  \leq \frac{r^2}{2}$,  we have $\frac{1}{\rho} \leq \frac{2}{r^2}$. Then \eqref{eq5-15} reduces to
\[ - C \sigma_1 + ( 2 \alpha - 2 ) \frac{u^2}{(\nu^{n+1})^2} (\sigma_2 \sigma_1 - 3 \sigma_3) \,\leq 0. \]
Choosing $\alpha$ sufficiently large we obtain an upper bound for $\kappa_1$.

\vspace{4mm}

Next, we consider the cases when $|x|^2 \geq \frac{r^2}{2}$, which implies $\rho \leq \frac{r^2}{2}$. We observe that
\begin{equation} \label{eq5-17}
\rho_i = - \frac{2}{u} \Big(   {\bf x} - ( {\bf x} \cdot \partial_{n+1}) \, \partial_{n+1} \Big)  \cdot \tau_i = \,- \frac{2}{u}  \sum\limits_{j = 1}^n ( {\bf x} \cdot \partial_{j}) \, (\partial_{j} \cdot \tau_i).
\end{equation}
Therefore,
\begin{equation} \label{eq5-16}
\begin{aligned}
\sum\limits_i \rho_i^2 = \,\,& \frac{4}{u^2} \sum\limits_{jk} ({\bf x} \cdot \partial_j) ({\bf x} \cdot \partial_k) \sum\limits_i (\partial_j \cdot \tau_i) (\partial_k \cdot \tau_i)
\\ = \,\,& 4 \sum\limits_{jk} ({\bf x} \cdot \partial_j) ({\bf x} \cdot \partial_k) \Big( \sum\limits_i \big(\partial_j \cdot \frac{\tau_i}{u} \big) \frac{\tau_i}{u} \Big) \cdot \partial_k
\\ = \,\,& 4 \sum\limits_{jk} ({\bf x} \cdot \partial_j) ({\bf x} \cdot \partial_k) \Big( \partial_j - (\partial_j \cdot \nu) \nu \Big) \cdot \partial_k
\\ \geq \,\,& 4 \Big( \sum\limits_j ({\bf x} \cdot \partial_j)^2 - \sum\limits_j ({\bf x} \cdot \partial_j)^2 \, \sum\limits_j (\partial_j \cdot \nu)^2 \Big)
\\ = \,\,& 4 \sum\limits_j ({\bf x} \cdot \partial_j)^2 (\nu^{n+1})^2 = \,4 u^2 |x|^2 (\nu^{n+1})^2 \geq \, 2\, r^2 u^2 (\nu^{n+1})^2.
\end{aligned}
\end{equation}

\vspace{3mm}

Case (ii):  if for some $ 2 \leq j \leq n$,  we have $|\rho_j| > d$, where $d$ is a small positive constant to be determined later.

By \eqref{eq4-6}, \eqref{eq5-12} and \eqref{eq5-8}, we have

\[ \begin{aligned}
\frac{h_{11j}}{\kappa_1 \,\ln \kappa_1} = - \frac{2 \,\rho_j}{\rho} + \Big( \beta \frac{ ({\bf x} \cdot \nu) (\tau_j \cdot \partial_{n+1}) - ({\bf x} \cdot \tau_j) \, \nu^{n+1} }{u(\nu^{n+1})^2} - 2 \alpha \frac{u (\tau_j \cdot \partial_{n+1})}{(\nu^{n+1})^3} \Big) \kappa_j.
\end{aligned} \]
It follows that
\[ \frac{h_{11j}^2}{\kappa_1^2 \,(\ln \kappa_1)^2} \geq \, \frac{2 \, \rho_j^2}{\rho^2} - C (\alpha + \beta)^2 \,\kappa_j^2 \geq \, \frac{ d^2}{\rho^2} + \frac{4 \, d^2}{r^4} - \frac{ C (\alpha + \beta)^2 }{\kappa_1^2} \geq \, \frac{ d^2}{\rho^2}\]
when $\kappa_1$ is sufficiently large.
Consequently, \eqref{eq5-15} reduces to
\[ - \frac{C \,\sigma_1}{\rho^2} + \frac{d^2}{20 \,\rho^2} \, \sigma_2^{jj} \,\ln \kappa_1 \, \leq 0. \]
Since $\sigma_2^{jj} \geq \frac{9}{10} \,\sigma_1$ when $\kappa_1$ is sufficiently large, we obtain an upper bound for $\kappa_1$.

\vspace{4mm}

Case (iii):  if $|\rho_j| \leq d$ for all $ 2 \leq j \leq n$, from \eqref{eq5-16} we can deduce that $|\rho_1| \geq c_0 > 0$.
By \eqref{eq4-6}, \eqref{eq5-12} and \eqref{eq5-8}, we have
\begin{equation} \label{eq5-18}
\frac{h_{111}}{\kappa_1\, \ln \kappa_1} = \, \frac{ \beta \kappa_1 \,b_1}{ (\nu^{n + 1})^2} \,  - \frac{2 \,{\rho}_1}{\rho} - \frac{2 \alpha u \kappa_1 (\tau_1 \cdot \partial_{n+1}) }{(\nu^{n+1})^3},
\end{equation}
where
\[ \begin{aligned}
b_1 =  \,\,& ({\bf x} \cdot \nu) \, \Big(\frac{\tau_1}{u} \cdot \partial_{n+1} \Big)  - \Big({\bf x} \cdot \frac{\tau_1}{u} \Big) \,\nu^{n+1}
\\ = \,\,&  \frac{\nu^{n+1}}{2} \,\rho_1 + \Big( \frac{\tau_1}{u} \cdot \partial_{n+1} \Big) \Big( {\bf x} \cdot \big( \nu - (\nu \cdot \partial_{n+1}) \partial_{n+1}  \big) \Big)
\\ = \,\,& \frac{\nu^{n+1}}{2} \,\rho_1 + \frac{1}{\nu^{n+1}} \Big( \frac{\tau_1}{u} \cdot \partial_{n+1} \Big) (\nu \cdot \partial_{n+1})  \sum\limits_i (\nu \cdot \partial_{i}) ( {\bf x} \cdot \partial_{i} )
\\ = \,\,& \frac{\nu^{n+1}}{2} \,\rho_1 + \frac{1}{\nu^{n+1}} \sum\limits_i \Big( \big( \frac{\tau_1}{u} \cdot \partial_{n+1} \big) \partial_{n+1} \Big) \cdot \Big( (\partial_i \cdot \nu) \nu \Big)( {\bf x} \cdot \partial_{i} )
\\ = \,\,& \frac{\nu^{n+1}}{2} \,\rho_1 + \frac{1}{\nu^{n+1}} \sum\limits_i \Big( \frac{\tau_1}{u} - \sum\limits_j \big(\frac{\tau_1}{u} \cdot \partial_{j} \big) \partial_{j} \Big) \cdot \Big( \partial_i - \sum\limits_k \big(\partial_i \cdot \frac{\tau_k}{u}\big) \frac{\tau_k}{u} \Big)( {\bf x} \cdot \partial_{i} )
\\ = \,\,& \frac{\nu^{n+1}}{2} \,\rho_1 + \frac{1}{\nu^{n+1}} \sum\limits_i \Big( - \frac{\tau_1}{u} \cdot \partial_i + \sum\limits_{jk} \big(\frac{\tau_1}{u} \cdot \partial_j \big) \big( \partial_i \cdot \frac{\tau_k}{u} \big) \big( \partial_j \cdot \frac{\tau_k}{u} \big) \Big)( {\bf x} \cdot \partial_{i} )
\\ = \,\,& \frac{\nu^{n+1}\,\rho_1}{2} + \frac{\rho_1 }{2 \,\nu^{n+1}} -  \frac{1}{2 \,\nu^{n+1}}    \sum\limits_{jk} \big(\frac{\tau_1}{u} \cdot \partial_j \big)  \big( \partial_j \cdot \frac{\tau_k}{u} \big)  \,\rho_k.
\end{aligned}\]
Note that in the last equality we have applied \eqref{eq5-17}.
Hence
\[ |b_1| \geq  \frac{\nu^{n+1}}{2} \,|\rho_1|  - \frac{1}{2 \,\nu^{n+1}} \sum\limits_{k \neq 1} |\rho_k| \,\geq c_1 > 0 \]
and \eqref{eq5-18} can be estimated as
\[ \Big| \frac{h_{111}}{\kappa_1 \ln \kappa_1} \Big| \geq \,\frac{\beta c_1 \,\kappa_1 }{2 (\nu^{n+1})^2} - \frac{C}{\rho} \geq \, \frac{\beta c_1 \,\kappa_1 }{4 (\nu^{n+1})^2}\]
when $\beta >> \alpha$ and $\kappa_1 \rho$ is sufficiently large. Taking this into \eqref{eq5-15} and observing that
\[ \sigma_2^{11} \kappa_1^2 \geq \frac{9}{10 \,n} \sigma_2 \,\sigma_1 \]
as $\kappa_1$ is sufficiently large,
we then obtain an upper bound for $\rho^2 \ln \kappa_1$.

\end{document}